\documentclass[a4paper,11pt,reqno, english]{amsart}  
\usepackage[utf8]{inputenc}
\usepackage[T1]{fontenc}
\usepackage{amsmath}
\usepackage[margin=1in]{geometry}
\usepackage{relsize}
\usepackage{float}
\usepackage[normalem]{ulem}

\usepackage{amsthm}

\usepackage{framed}

\usepackage{amssymb}
\usepackage{tikz}
\usepackage{enumerate}
\usepackage[mathscr]{euscript}
\usepackage{url}

\usepackage[markcase=noupper
]{scrlayer-scrpage}
\ohead{}
\newcommand{\dgh}{d_{\mathrm{GH}}}
\newcommand{\dB}{d_{\mathrm{B}}}
\newcommand{\dH}{d_{\mathrm{H}}}
\newcommand{\cset}{\mathrm{K}}
\newcommand{\dgm}{\mathrm{dgm}}
\newcommand{\pow}{\mathrm{pow}}
\newcommand{\diam}{\mathrm{diam}}
\newcommand{\ecc}{\mathrm{ecc}}
\newcommand{\dis}{\mathrm{dis}}
\newcommand{\ud}{\mathrm{U}}
\newcommand{\hd}{\mathrm{H}}

\newcommand{\id}{\text{id}}

\DeclareMathOperator{\im}{im}
\DeclareMathOperator{\ult}{ult}
\DeclareMathOperator{\hyp}{hyp}

\newcommand{\R}{\mathbb{R}}

\newcommand{\N}{\mathbb{N}}
\newcommand{\M}{\mathcal{M}}
\newcommand{\F}{\mathcal{F}}

\newcommand{\D}{\mathcal{D}}
\newcommand{\V}{\mathcal{V}}

\theoremstyle{plain}
\newtheorem{thm}{Theorem}[section]

\newtheorem{cor}[thm]{Corollary}

\newtheorem*{thm*}{Theorem}

\theoremstyle{definition}

\newtheorem{defn}[thm]{Definition}

\newtheorem{prop}[thm]{Proposition}
\newtheorem{ex}[thm]{Example}
\newtheorem{lemma}[thm]{Lemma}

\newtheorem{rem}[thm]{Remark}

\usepackage{algorithm}
\usepackage[noend]{algpseudocode}
\def\BState{\State\hskip-\ALG@thistlm}

\title{New families of simplicial filtration functors}

\author{Samir Chowdhury}
\email[S. Chowdhury]{chowdhury.57@osu.edu}
\address{The Ohio State University, Mathematics.}

\author{Nathaniel Clause}
\email[N. Clause]{nathaniel.clause@vanderbilt.edu}
\address{Vanderbilt University.}

\author{Facundo M\'emoli}
\email[F. M\'emoli]{memoli@math.osu.edu}
\address{The Ohio State University, Mathematics and Computer Science and Engineering.}

\author{Jose \'Angel S\'anchez}
\email[J. A. S\'anchez]{jose.sanchez@cimat.mx}
\address{University of Guanajuato}

\author{Zoe Wellner}
\email[Z. Wellner]{zaw5@cornell.edu}
\address{Cornell University.}

\begin{document}
\begin{abstract}

The so called \v{C}ech and Vietoris-Rips simplicial filtrations are designed to capture information about the topological structure of metric datasets. These two filtrations are two of the workhorses in the field of topological data analysis. They enjoy stability with respect to the Gromov-Hausdorff (GH) distance, and this stability property allows us to estimate the GH distance between finite metric space representations of the underlying datasets.

Via the concept of Gromov's curvature sets we construct a rich theoretical framework of valuation-induced stable filtration functors. This framework includes the \v{C}ech and Vietoris-Rips filtration functors as well as many novel filtration functors that capture diverse features present in datasets. We further explore the concept of basepoint filtrations functors and use it to provide a classification of the filtration functors that we identify.

\end{abstract}

\maketitle

\section{Introduction}
When  analyzing a  dataset one usually  regards the data set as a finite metric space of some sort. When looking at a finite metric space, however, there is a priori no interesting topology. In order to be able to uncover the topology features which may be present but not apparent in a data set and to be able to gain information from this topology, we want to somehow find a structure that allows us to place higher dimensional structures on top of the data points.

A way to do so is to compute a filtration such as the Vietoris-Rips  or the \v{C}ech filtration \cite{edels-book}. The resulting object is a nested sequence of simplicial complexes for which we can compute \emph{persistent homology} \cite{frosini1992measuring, robins1999towards, carlsson2009topology, edelsbrunner2014persistent}. The goal of this paper is to find filtrations that are different from the \v{C}ech and Vietoris-Rips filtrations, as part of a larger program of characterizing \emph{all} the possible filtration functors acting on metric spaces. 

Why do we search for alternative filtration functors when the Vietoris-Rips and \v{C}ech functors are already well-studied? A filtration can be viewed as a projection of the underlying data: we give up some knowledge to obtain a simpler representation of the data, which we then linearize (via homology) in order to obtain persistent vector spaces and readily interpretable barcodes/persistent diagrams. Different filtrations yield different projections of the original data, so visualizing the barcodes obtained from different filtrations enables us to have a more complete representation of the data. This is the main motivation behind our work.

\begin{figure}
  \begin{minipage}[c]{0.2\textwidth}
   \begin{tikzpicture}[scale=0.4]

    \filldraw[thick, fill=black!20!white](-2,-1)
    to [out=10,in=190] (1,-2.5)
    to [out=10,in=90] (1,-4) 
    to [out=-90,in=30] (-0.5,-4)
    to [out=180, in=180] (-2,-1);

    \filldraw[thick, fill=black!20!white](6,-1.5)
    to [out=-60,in=190] (6,-3)
    to [out=10,in=90] (6.5,-4)
    to [out=-100,in=-140] (3.5,-4)
    to [out=0,in=-60] (4,-2)
    to [out=120, in=120] (6,-1.5);

    \filldraw[thick, fill=black!20!white](-3.5,4.5)
    to [out=10,in=190] (-2,3)
    to [out=10,in=90] (-1.25,1.5) 
    to [out=-90,in=30] (-3.5,1.5)
    to [out=180, in=180] (-3.5,4.5);

    \filldraw[thick, fill=black!20!white](6,2.2)
    to [out=10,in=190] (7.5,4.2)
    to [out=10,in=90] (9,3.2) 
    to [out=-90,in=30] (7,1.2)
    to [out=180, in=180] (6,2.2); 
    
    \draw[->, ultra thick] (-1.2,2.8) to (6.35,2.8);
    \draw[->, ultra thick] (-3.15,1.45) to (-1.9,-0.8);
    \draw[->, ultra thick](1.5,-3) to (4,-3);
    \draw[->, ultra thick] (5.25,-1) to (6.65,1.15);
    
    \node at (-3.45,3.5){\Large $\mathcal{M}$};
    \node at (-1.75,-1.75){\Large $\mathcal{F}$};
    \node at (5,-2){\Large $\mathcal{V}$};
    \node at (7.75,3.5){\Large $\mathcal{D}$};

    \node at (2,3.55){\large $\dgm_k^{\ast}$};
    \node at (2.6,-3.7){\large $H_k$};
    \node at (7,-0.5){\large $\dgm$};
        
    \draw[->, ultra thick, color=red] (-3.15,1.45) to (-1.9,-0.8);

    \draw[->, thick, color=red] (-3.75,1.45)
    to [out=240, in=140] (-2.75,-1.2);
    
    \draw[->, thick, color=red] (-3.5,1.45)
    to [out=240, in=120] (-2.5,-1);
    
    \draw[->, thick, color=red] (-3.35,1.45)
    to [out=270, in=120] (-2.2,-0.8);
    
    \draw[->, thick, color=red] (-2.75,1.45)
    to [out=-60, in=100] (-1.5,-0.8);
    
    \draw[->, thick, color=red] (-2.5,1.4)
    to [out=-40, in=80] (-1.25,-1);

\end{tikzpicture}
  \end{minipage}\hfill
  \begin{minipage}[c]{0.7\textwidth}
   \caption{\footnotesize{The diagram shows the main processes that metric datasets undergo when applying persistent homology. 
  The  origin of such a pipeline is the space of all finite metric spaces $\M$. We want our maps to end in $\D$, the space of diagrams/barcodes. We first map to the space of all filtered spaces $\F$ using a given filtration functor. The deep-red arrow represents the usual filtration functors used (\v{C}ech or Vietoris-Rips.) We want to find as many alternative red arrows as possible. From $\F$ we map to persistent vector spaces, $\V$ via the homology functor, $\mathrm{H}_k$ (with field coefficients). From $\V$ the $\dgm$ map produces the persistence barcode. The $\dgm_k^{\ast}$ map represents the overall process that we have split into three parts.}}
  \label{fig:overview}
  \end{minipage}
\end{figure}
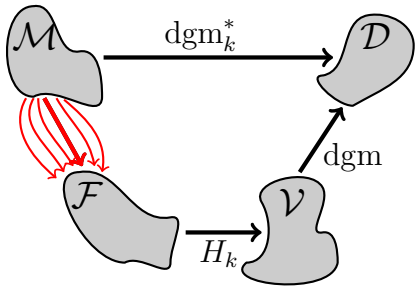

\newpage
\tableofcontents

\section{Preliminaries}
For a set $X$, we let $\pow(X)$ denote the set of all finite non-empty subsets of $X$. By $\mathcal{M}$ we will denote the collection of all finite metric spaces. The Hausdorff distance between closed subsets of a metric space will be denoted by $d_{\mathrm{H}}$. Given a finite metric space $(X,d_X)$ we consider the map $\iota_X:\pow(X)\rightarrow \mathcal{M}$ given by $\sigma\mapsto (\sigma,d_X|_{\sigma\times \sigma})$, that is $\iota_X$ takes a  subset of $X$ and endows it with the restriction of the metric of $X$.
By the \emph{diameter} $\diam(X)$ of a finite metric space $(X,d_X)$ we mean the number $\max_{x,x'\in X}d_X(x,x')$, and by the \emph{eccentricity function associated to $(X,d_X)$} we mean the function $\ecc_X:X\rightarrow \R_+$ such that $x\mapsto \max_{x'\in X}d_X(x,x').$

Throughout this article, we will use homology with field coefficients. We will write $\mathrm{H}_k$ to denote the $k$th homology functor.

Given two finite metric spaces $(X,d_X)$ and $(Y,d_Y)$ and any non-empty relation $R\subset X\times Y$ we consider its \emph{distortion} \cite{burago} given by $$\dis(R):=\max\limits_{(x,y),(x',x')\in R}|d_X(x,x')-d_Y(y,y')|.$$ A particular class of relations between sets $X$ and $Y$ is given by \emph{correspondences}: these are relations $R$ such that the canonical projection maps are surjective. We denote as $\mathcal{R}(X,Y)$ the set of all correspondences between $X$ and $Y$. The \emph{Gromov-Hausdorff distance} between $X,Y\in\mathcal{M}$ is defined as $\dgh(X,Y):=\frac{1}{2}\min\limits_{R}\dis(R)$ where $R$ ranges over all correspondences between $X$ and $Y$.

\defn[\cite{gromov}]{Given a metric space $(X,d_X)$ and a natural number $n$ consider the map $$D_X^{(n)}:X\times \cdots\times X\rightarrow \R^{n\times n}$$ given by 
$$(x_1,\ldots,x_n)\mapsto \big(d_X(x_i,x_j)\big)_{i,j=1}^n.$$

Then, the \emph{$n$-th curvature set} of $(X,d_X)$ is the set 

$$\cset_n(X):=\im(D_X^{(n)}).$$}
\begin{rem}\label{rem:functorial-cset}
Curvature sets enjoy the following kind of functoriality: suppose $X\hookrightarrow Y$ isometrically, then, for any $n\in \N$ one has 
$$\cset_n(X)\subseteq \cset_n(Y).$$
\end{rem}

It turns out that curvature sets are stable with respect to the Gromov-Hausdorff distance. In fact, curvature sets of order $n$ are subsets of $\R^{n\times n}$, which we endow with the $\ell_\infty$ distance. We stress that we will be considering finite subsets of $\R^{n\times n}$. For the rest of this paper, when we write $\pow(\R^{n\times n})$ or $\pow(\R)$ we are always referring to the space of all finite subsets of $\R^{n\times n}$ and $\R$, respectively. Then, we have 
\begin{thm}[\cite{dgh-props}]\label{thm:cset}
For any pair of compact metric spaces, and all $n\in\N$,  $$d_{\mathrm{H}}^{(\R^{n\times n}, \ell_\infty)}(\cset_n(X),\cset_n(Y))\leq 2 \cdot\dgh(X,Y).$$
\end{thm}

\subsection{Filtered spaces and their persistence diagrams}

Although curvature sets define a lower bound on the Gromov-Hausdorff distance, they fall short in terms of being a way of understanding the underlying spaces. They lack clear interpretability since, as $n\in\N$ grows, $\cset_n(\cdot)$ is embedded in a high-dimensional space. We want to find other invariants of spaces that can be visualized and interpreted, and at the same time provide bounds on Gromov-Hausdorff distance. The persistence diagrams of Vietoris-Rips and \v{C}ech filtration functors offer both a bound and an interpretation. To capture this notion, we craft a general definition of filtrations on spaces.

\defn{ Let $X$ be a finite set. A \emph{filtration on $X$} is any map $\phi_{X}:\pow(X) \rightarrow \R_+$ that satisfies the monotonicity condition: 
$$\phi_X(\tau)\leq \phi_X(\sigma), \hspace{1cm} \forall \, \, \tau\subset \sigma\subset X. $$ 
Any pair $(X,\phi_X)$ where $X$ is a finite set and $\phi_X$ is a filtration over $X$ will be called a \emph{filtered space}. The collection of all such pairs will be denoted by $\mathcal{F}.$
}

To interpret this definition, we will consider the following. Let $X\in \M$ and $\Phi_X$ be a filtration on $X$. For all $t\in \R_+$, we define the following:
$$\Phi_X[t]= \{ \sigma\subset X: \Phi_X(\sigma)\leq t\}.$$
We call $\Phi_X[t]$ the filtration space at time $t$. 

The monotonicity condition is required for $\Phi_X [t]$ to be a simplicial complex at any $t\geq 0$. As $t \rightarrow \infty$, this complex $\Phi_X[t]$ grows until it becomes the full simplex $\pow(X)$.


We start defining an empty simplicial complex (no simplicies in it). We will add simplicies to it through time. This simplicial complex will have the points of $X$ as its 0-simplicies. The function $\phi_X$ is giving us the time of arrival of any simplex $\sigma\subset X$ into the space. The monotonicity condition is required for all faces of any simplex to be added before the simplex. This definition of filtrated spaces is restricted to growing simplicial complexes. 


A filtered space gives rise to a persistence diagram in the manner that we describe now. We start with a filtration $\{\Phi_X[t] \subseteq \Phi_X[t']\}_{0 \leq t \leq t'}$. Applying the $k$th homology functor $\mathrm{H}_k: \F \rightarrow \V$ transforms this filtration along with its inclusion maps into a sequence of vector spaces with linear maps, i.e. a \emph{persistent vector space}.

More generally, a persistent vector space is defined to be a family of vector spaces $\{U^\delta \xrightarrow{\mu_{\delta ,\delta'}} U^{\delta'}\}_{\delta \leq \delta'\in \R}$ such that: (1) $\mu_{\delta,\delta}$ is the identity for each $\delta \in \R$, and (2) $\mu_{\delta,\delta''} = \mu_{\delta',\delta''}\circ \mu_{\delta,\delta'}$ for each $\delta \leq \delta' \leq \delta'' \in \R$. To each persistent vector space, it is possible to associate a full invariant called a \emph{persistence diagram} or \emph{persistence barcode} \cite{carlsson2005persistence}. We let $\dgm: \V\rightarrow \D$ denote the \emph{diagram map} that maps a persistent vector space to its barcode.

\subsection{Filtrations and their stability}
It is possible to define a distance $d_\mathcal{F}$ between filtered spaces as follows \cite{tripods}: given $(X,\phi_X)$ and $(Y,\phi_Y)$ in $\mathcal{F}$ let

$$d_{\mathcal{F}}((X,\phi_X),(Y,\phi_Y)):= \inf_{Z,\pi_X,\pi_Y} \max_{\sigma \subset Z} |\phi_X(\pi_X(\sigma)) - \phi_Y(\pi_Y(\sigma ))|, $$
where $Z$ is any finite set $\pi_X$ and $\pi_Y$ are surjective maps from $Z$ to $X$ and $Y$, respectively. There is a result that can help us with the notion of stability:

\begin{thm}[\cite{tripods}]\label{thm:tripods}
For all finite filtered spaces $(X,\phi_X)$ and $(Y,\phi_Y)$, and all $k\in\mathbb{N}$ one has $$ \dB( \dgm_k^{\phi_X}(X) , \dgm_k^{\phi_Y}(Y) )\leq d_{\mathcal{F}}((X,\phi_X), (Y,\phi_Y)).$$
\end{thm}

\subsection{Filtration functors and their stability}

Theorem \ref{thm:tripods} provides a nice stability result between the category of filtered spaces and the category of diagrams. Observe from Figure \ref{fig:overview} that the complementary portion of the persistent homology pipeline consists of the filtration maps from finite metric spaces to filtered spaces. These filtration maps are the focus of our work. We want to both create new and rich families of filtration maps, and find stability results between the category of filtered spaces and the category of finite metric spaces. Towards this end, let us consider the following definitions:

\defn{\label{def:filtfunct}
A \emph{filtration functor} is a map $\Phi: \M\rightarrow \F$ such that for every $X\in\mathcal{M}$, $\Phi_X$ is a filtration over $X$.}

The sense of functoriality may change depending on the category we want to consider on $\M$. We can first consider $\M^{\text{Lip}}$ to be the category of finite metric spaces with 1-Lipschitz maps between them. Later on, we will develop filtrations that are not functorial with respect to 1-Lipschitz maps, but which are nonetheless functorial with respect to isometries.\footnote{See \cite{clust-um} for similar categories appearing in work related to classifying clustering functors.} 

\defn{
We say that a filtration functor $\Phi$ is functorial on $\M^{\text{Lip}}$ if for all pairs $(X,d_X),(Y,d_Y)\in \M$ and any 1-Lipschitz map $h: X\rightarrow Y$, $\Phi_X(\sigma)\geq \Phi_Y(h(\sigma))$ for all $\sigma\subseteq X$.\\

Functoriality can be thought in the following terms. Let $X,Y\in\M$ and $h:X\rightarrow Y$ be a 1-Lipschitz map. For each $t\in \R_+$ the function $h$ induces a map between the filtered spaces $h:\Phi_X[t] \rightarrow \Phi_Y[t]$. From the functoriality condition, it can be shown that $h$ is a simplicial map for all $t\in \R_+$.

We now reframe the well known \v{C}ech and Vietoris-Rips filtration functors in terms of this new definition.

\begin{ex}
We define the \v{C}ech filtration functor to be a function given by,
$$C:\M\rightarrow \F$$
$$ X\in\M \mapsto C_X,$$
where $C_X: \pow(X)\rightarrow \R$ is defined by,
$$ C_X(\sigma):= \min_{p\in X} \max_{x\in \sigma} d_X(p,x).$$
\end{ex}

\begin{ex}
We define the Vietoris-Rips filtration functor to be $ R:\M\rightarrow \F$ defined as, $X\mapsto R_X$, where,
$$ R_X(\sigma):= \diam(\iota_X(\sigma)),  \hspace{1cm} \forall \, X\in \M, \, \sigma\subset X.$$
\end{ex}

In the most basic case, a filtration functor can be any rule that assigns to each finite metric space a filtration. These choices on each space are not required to satisfy any condition with respect to other spaces. In order to find relationships between the diagrams we obtain through filtration functors, we have to impose a regularity condition on them. 

\defn{We say that a filtration functor $\Phi:\M\rightarrow \mathcal{F}$ is \emph{stable} if there exists $L\geq 0$ such that:
$$ d_{\mathcal{F}}((X,\Phi_X),(Y,\Phi_Y))\leq L\cdot \dgh(X,Y),$$ for all 
$X,Y\in\M.$ If the above condition holds for a given $L\geq 0$ we will call $\Phi$ a \emph{$L$-stable filtration functor.}
}

Given any filtration functor $\Phi$ and $k\in\N$, by $\dgm_k^\Phi$ we will denote the composite map $$\dgm\circ \mathrm{H}_k\circ \Phi:\mathcal{M}\rightarrow \mathcal{D}.$$

Combining the above one has the following: 
\begin{cor}[\cite{tripods}]
Let $L\geq 0$ and $\Phi:\mathcal{M}\rightarrow \mathcal{F}$ be any $L$-stable filtration functor. Then, for all $k\in \N$ and all $X,Y\in\mathcal{M}$,
$$d_{\mathrm{B}}\big(\dgm_k^\Phi(X),\dgm_k^\Phi(Y)\big)\leq L\cdot \dgh(X,Y).$$
\end{cor}

\section{A theorem about existence of rich families}

One could try to compute the Gromov-Hausdorff distance by finding a filtration functor $\Phi$ such that the map $\dgm^{\Phi}_k$ is an isometric embedding of $\M$ into $\D$. This would mean that the Gromov-Hausdorff distance between two particular spaces equals the bottleneck distance between their diagrams. Such a filtration does not exist when we restrict our scope to filtrations sharing properties with Rips or \v{C}ech. By this, we mean we will consider filtrations such that all 0-simplexes will be added at time zero.

\begin{prop}
 Let $\Phi$ be any non-trivial  $1$-stable filtration functor such that,
	$$ \Phi_X(\{x\})= 0, \hspace{1cm} \forall\, x\in X,\,\forall X\in\mathcal{M}.$$ 
 Let $k\in\N$. Then, there exist two different finite metric spaces $X$ and $Y$ such that $$d_\mathrm{B}(\dgm_k^\Phi(X),\dgm_k^\Phi(Y))<\dgh(X,Y).$$
\end{prop}

\begin{proof}
We divide the proof into two cases: $k\geq 1$ and $k=0$.

\noindent
Let $k\geq 1$ and $\Phi:\M \rightarrow \F$ be a $1-$stable filtration functor such that:
$$ \dgh(X,Y)= \dB(\dgm^\Phi_k(X) , \dgm^\Phi_k(Y)), \hspace{1cm} \forall \, X,Y \in \M.$$
It follows that for $X\in \M$, if $*\in \M$ represents the single point metric space, then
$$\frac{\diam(X)}{2}= \dgh(X,*)= \dB(\dgm^\Phi_k(X) , \dgm^\Phi_k(*)).$$
Since $*\in \M$ has just one point and $k\geq 1$, we see $\dgm^\Phi_k(*)=\{\emptyset\}$ by the fact that a simplicial complex with just one point has no homology of dimension greater than 0. This claim is also true for a two-point space. Then, if $X\in \M$ is the two-point space with distance 1 between its points:
$$\frac{1}{2}= \dgh(*,X)= \dB(\dgm^\Phi_k(*),\dgm^\Phi_k(X))=\dB(\emptyset,\emptyset)= 0. $$
This is a contradiction, so such a $\Phi$ does not exist.

\noindent
For $k=0$, let $\Phi$ be a filtration functor such that,
	$$\dgh(X,Y)= \dB(\dgm^{\Phi}_0(X),\dgm^{\Phi}_0(Y)),\hspace{1cm} \forall \, X,Y\in \M. $$
Let $r\geq 0$. Let $X_r=\{0,r\}\subset \R$, be the space of two points at distance $r$ from each other. We know that $\Phi_{X_r}(\{0\})=\Phi_{X_r}(\{r\})= 0$. Then, there exists $f(r)>0$ such that,
$$\dgm^{\Phi}_0(X_r)= \{[0,\infty), [0,f(r) ) \}.$$
Also, it is clear that $\dgm^{\Phi}_0(*)=\{[0,\infty)\}.$
From this, we can see that,
\begin{eqnarray*}
\dgh(X_r,*) &=&\dB(\dgm^{\Phi}_0(X),\dgm^{\Phi}_0(*) )\\
    &=& \dB(\{[0,\infty), [0,f(r))\}, \{[0,\infty) \}) \\
    &=& \frac{f(r)}{2}.
\end{eqnarray*}
We know that $\dgh(X,*)=\frac{\diam(X)}{2}$ for all spaces $X\in \M$. From this, we can see that $\dgh(X_r,*)=\frac{r}{2}$. It follows that $f(r)= r$.
Now, if $0< \varepsilon < 1$, we can see that, since $X_{1}$ and $X_{1+\varepsilon}$ are homothetical spaces \cite[Example 3.3]{dgh-props},
$$\dgh(X_{1},X_{1+\varepsilon})= \frac{\varepsilon }{2}. $$
But it also follows that,
\begin{eqnarray*}
	\dB(\dgm^{\Phi}_0(X_1),\dgm^{\Phi}_0(X_{1+\varepsilon}))&=& \dB(\{[0,\infty), [0,1)\},\{[0,\infty), [0,1+\varepsilon)\})\\
    &=& \varepsilon\neq \frac{\varepsilon}{2}.
\end{eqnarray*}
We conclude from this contradiction that such filtration functor $\Phi$ does not exist. 
\end{proof}

We maintain our motivation of finding a filtration functor for which the bottleneck distance between corresponding barcodes of any two given finite metric spaces attains the Gromov-Hausdorff distance between them. This has been unsuccessful when considering any single filtration, so instead we ask: is it possible to find a \emph{family} $(\Phi_\alpha)_{\alpha}$ of $1$-stable filtration functors such that for every pair of spaces $X,Y\in \M$ and $\epsilon>0$ there is a functor $\Phi_\alpha$ in this family such that this particular functor attains the Gromov-Hausdorff distance (up to an error of magnitude at most $\epsilon$) through the bottleneck distance? This idea is developed in the following theorem. Furthermore, we ask and succeed in proving that this is true only when considering barcodes corresponding to $k=0$.

\begin{thm}
There exists $\mathfrak{F}$, a family of $1$-stable filtration functors such that for all $X,Y \in \M$, we have 
$$\dgh(X,Y) = \sup\limits_{\Phi \in \mathfrak{F}}\dB\big(\dgm_0(\Phi_X),\dgm_0(\Phi_Y)\big).$$
\end{thm}

\begin{rem}
Note that the theorem indicates that in order to fully capture geometric dissimilarity between finite metric spaces it is enough to only consider the case of $0$-dimensional barcodes. We however remark that this result is merely an existence result: it does not offer guidance in terms of how the family $\mathfrak{F}$ should be chosen. 
\end{rem}

\begin{proof}
Consider the family $\mathfrak{F}=\{\Upsilon^{(Z)}:\mathcal{M}\rightarrow \mathcal{F},\,Z\in\mathcal{M}\}$ of filtration functors indexed by $Z\in\mathcal{M}$. In other words, our family contains a filtration functor $\Upsilon^{(Z)}$ for each $Z\in \mathcal{M}$. For each $Z\in\M$ we  define the  filtration functor $\Upsilon^{(Z)}:\M\rightarrow \F$ given by $X\mapsto \Upsilon^{(Z)}_X$ where,
$$ \Upsilon^{(Z)}_X(\sigma) := \dgh(X,Z), \hspace{1cm} \forall \sigma\subset X \in \M. $$

This filtration is constant on simplices and at the simplicial level yields an empty simplicial complex on the interval $[0,\dgh(X,Z))$ and the complete simplicial complex, i.e. the power set of $X$ in the interval $[\dgh(X,Z),\infty)$. Thus, $$\dgm_0\big(\Upsilon^{(Z)}_X\big)=\left\{[\dgh(X,Z),\infty)\right\}.$$
Thus, $$\dB\big(\dgm_0(\Upsilon^{(Z)}_X),\dgm_0(\Upsilon^{(Z)}_Y)\big)=|\dgh(X,Z)-\dgh(Y,Z)|.$$ 

From this, and the triangle inequality on $\M$ it follows that,
$$\dB\big(\dgm_0(\Upsilon^{(Z)}_X),\dgm_0(\Upsilon^{(Z)}_Y)\big)=|\dgh(X,Z)-\dgh(Y,Z)|\leq \dgh(X,Y),$$
for all $X,Y,Z \in \M$ which proves that all functors in $\mathfrak{F}$ are $1$-stable.

Since $\dgh(Y,Y)=0$, it follow that 
$$ \dB(\dgm_0(\Upsilon^{(Y)}_X),\dgm_0(\Upsilon^{(Y)}_Y))= \dgh(X,Y) \hspace{1cm} \forall X,Y \in \M.$$
We conclude that,
$$\dgh(X,Y)= \sup_{Z\in\M}\dB\big(\dgm_0(\Upsilon^{(Z)}_X),\dgm_0(\Upsilon^{(Z)}_Y)\big) \hspace{1cm} \forall X,Y \in \M.$$
\end{proof}

\begin{rem}
See \cite[Theorem 15]{Frosini2016} for a related result in the category of topological spaces.
\end{rem}

So now we have a family that satisfies the desired property. The problem we face is that the construction shown needs previous knowledge of the Gromov-Hausdorff distance. This indicates that the family we built in the proof is not useful in terms of computations. The theorem does, however, guarantee the existence of families of this nature. Such a theorem suggests identifying rich families of stable filtration functors.

\section{Local and Global Filtration Functors}\label{sec:valuations}

\subsection{Valuations}
The $n-$th curvature set of a metric space contains all the metric information of $n$-tuples of points in the space. We can use this information to decide the time of arrival of a simplex in a filtration. If we want to use information from $\cset_n$, what we now need is a rule that assigns a real number to elements of $\pow(\R^{n\times n})$.
This can be done via the notion of valuations:

\defn[Valuation]{\label{def:valuations} Given $n\in\mathbb{N}$,  an $n$-\emph{valuation} is any map $\nu_n:\pow(\R^{n\times n})\rightarrow \R_+$ such that $\nu_n$ is \emph{monotonic}, meaning that $\nu_n(A)\geq \nu_n(B)$ for all $B\subset A \in \pow(\R^{n\times n})$. We will denote by $\mathfrak{V}_n$ the set of all $n$-valuations.}

Our use of the term valuation deviates from the usual meaning: we do not assume the modularity property.

We are interested in defining filtrations through valuations. The monotonicity condition is imposed so the  simplicial construction induced from a valuation defines a simplicial complex. It ensures that  if a simplex is in the filtered space at a given time, all sub-simplices are also in the filtered space at that time. 

\defn[Filtration functor induced by a valuation]{Given $n\in\N$ and any $\nu_n\in\mathfrak{V}_n$ we  induce the filtration functor $\Phi^{\nu_n}:\mathcal{M}\rightarrow \mathcal{F}$ defined as follows: for any $X\in \mathcal{M}$ and any $\sigma\subseteq X$, $$\Phi_X^{\nu_n}(\sigma):=(\nu_n \circ \cset_n\circ \iota_X)(\sigma).$$
We refer to $\Phi^{\nu_n}$ as the \emph{filtration functor induced by $\nu_n$}.}

\begin{lemma}
$\Phi^{\nu_n}$ is well defined.
\end{lemma}
\begin{proof}
Let  $\nu_n\in\mathfrak{V}_n$ be a valuation. Let $X\in \M$ and $\tau\subset\sigma\subset X$. By Remark \ref{rem:functorial-cset} $\tau\subset \sigma$ implies $\cset_n(\iota_X(\tau))\subset \cset_n(\iota_X(\sigma))$. Since $\nu_n$ is monotonic, 
	$$\Phi^{\nu_n}_X(\tau)= \nu_n(\cset_n(\iota_X(\tau)))\leq \nu_n(\cset_n(\iota_X(\sigma)))= \Phi^{\nu_n}_X(\sigma). $$
Thus, $\Phi^{\nu_n}_X$ is a filtration on $X$ and $\Phi^{\nu_n}$ is a filtration functor. It will be seen later that $\Phi^{\nu_n}$ is not necessarily functorial on $\M^{\text{Lip}}$. A further treatment of functoriality on local filtration functor is developed in \S \ref{sub:functcharact}.
\end{proof}

\defn{A filtration functor $\Phi:\M\rightarrow \F$ is \emph{local} if for some $n\in \N$ there exists a valuation $\nu_n \in\mathfrak{V}_n$ such that $\Phi = \Phi^{\nu_n}$, in which case we say that $\Phi$ is a \emph{$n$-local filtration functor}. If no such $n$ exists then we say that $\Phi$ is \emph{global}.}

\rem{\label{rem:valu} We make the following remarks: \begin{enumerate}
\item Notice that if the filtration functor $\Phi$ is $n$-local then it is $n'$-local for all $n'\geq n$. Thus, by convention, whenever we say a filtration functor is $n$-local, we are referring to the minimal $n$ for which this is true.
\item The Rips functor is $2$-local.  Indeed, this follows from noticing that for any $X\in\mathcal{M}$ and $\sigma\subset X$:
\begin{eqnarray*}
R_X(\sigma)&=& \diam(\sigma) \\
	&=& \max_{\{x,x'\}\subset \sigma} d_X(x,x') \\
    &=& \max_{A\in \cset_2(\iota_X(\sigma) )}\max_{i,j}A_{ij}.
\end{eqnarray*}
In words, the arrival time of a simplex under Rips depends only on pairwise distances between points in the simplex, which is information given by $\cset_2(\sigma)$. We can see that the $2$-valuation that generates the Rips filtration functor is $\nu_2(\mathcal{A}):= \max\limits_{A\in \mathcal{A}} ||A||_{\infty}$ for all $\mathcal{A} \in \pow(\R^{2\times 2})$.

 \item The \v{C}ech functor is not local. To show this, we construct a counterexample: Let $(P,d_P)$ be the metric space consisting of three equidistant points $P=\{p_1,p_2,p_3 \}$ at distance 2 from each other. Also, let $(Q,d_Q)$ be the  metric space consisting of four points $Q=\{q_0,q_1,q_2,q_3\}$ where the metric is given by the matrix:
$$d_Q= \begin{pmatrix}
	0 & 1 & 1 & 1 \\
    1 & 0 & 2 & 2 \\
    1 & 2 & 0 & 2 \\
    1 & 2 & 2 & 0 \\
\end{pmatrix}.$$

Since the subset $A=\{q_1,q_2,q_3\}\subset Q$ and $P$ are isometric, $\cset_n(\iota_Q(A))= \cset_n(P)$ for all $n\in \N$. If the \v{C}ech functor were $n$-local for some $n\in \N$, then, there would exist $\nu_n\in\mathfrak{V}_n$ such that
$$C_Q(A)= \nu_n(\cset_n(\iota_Q(A)))= \nu_n(\cset_n(P))= C_P(P).$$
But by definition of the \v{C}ech filtration functor $C_X(A)= 1$ and $C_Y(P)=2$, which yields a contradiction. The intuitive explanation why the \v{C}ech complex construction is not $n$-local for any $n\in \N$ is that the arrival time of a simplex depends on properties of the whole space, and not only on the metric properties of the points in the simplex. Thus $C$ is global.
\item Applying the Vietoris-Rips filtration functor on data is computationally better than \v{C}ech since the time of arrival of a simplex depends only on information about $2-$point comparisons of points in the simplex, instead of comparisons between all points each time. In the same sense, $n$-local filtrations functors for sufficiently small $n\in\N$ are computationally better than global filtrations such as \v{C}ech. 
\end{enumerate}}

We now restrict ourselves to a suitable class of valuations that yield stable filtration functors.

\defn{Let $n\in\mathbb{N}$ and let  $\nu_n \in \mathfrak{V}_n$. Given $L\geq 0$, we  say that $\nu_n$ is \emph{$L-$stable} if and only if for all non-empty  $A,B\subset \R^{n\times n}$:
 $$ |\nu_n(A)-\nu_n(B)|\leq L \cdot \,\dH^{(\mathbb{R}^{n\times n}, \ell_\infty)}( A, B ).$$
 We denote by $\mathfrak{V}_n^L$ the subset of all $L$-stable $n$-valuations.}

Stability is a desirable property since it relates the Gromov-Hausdorff  distance in $\M$ to the bottleneck distance on the space of barcodes. By invoking Theorem \ref{thm:tripods} and recalling the stability of the $\cset_n$ map (cf. Theorem \ref{thm:cset}),  we obtain the following:

\begin{thm}\label{lstablevaluation} Let $n\in\N$ and $L>0$. Then, for any $\nu_n \in\mathfrak{V}_n^L$ one has that for all $X,Y\in\M$ and $k\in\mathbb{N}$,
 $$ \dB\left( \dgm_{k}^{\Phi^{\nu_n}}(X),\dgm_{k}^{\Phi^{\nu_n}}(Y) \right)\leq 2L\cdot\dgh(X,Y). $$
\end{thm}
\begin{proof}
Let $X,Y\in \M$. Let $R\subset X\times Y$ be the correspondence that attains the Gromov-Hausdorff distance between this two spaces, $\dgh(X,Y)=\frac{1}{2}\dis(R)$. Now, we define the maps $\pi_1:R\rightarrow X$ and $\pi_2: R\rightarrow Y$ the natural projections from $R$ to $X$ and $Y$ respectively. From the fact that $R$ is a correspondence it follows that $\pi_1$ and $\pi_2$  are surjective. 

For this, first notice that the tripod $(R,\pi_1,\pi_2)$ consists of a set $R$, and two surjections from $R$ to $X$ and $Y$ respectively. Now, let $S\subset R$. It is worth noticing that $S$ is a correspondence between $\pi_1(S)$ and $\pi_2(S)$. From the fact that $\nu_n\in\mathfrak{V}_n^L$ and Theorem \ref{thm:cset}, 
\begin{eqnarray*}
|(\nu_n\circ \cset_n\circ\iota_X)(\pi_1(S))-(\nu_n\circ \cset_n\circ \iota_Y)(\pi_2(S))|&\leq & L \cdot\dH(\cset_n(\iota_X(\pi_1(S))), \cset_n(\iota_Y(\pi_2(S))))\\
  &\leq& 2L\cdot\dgh(\iota_X(\pi_1(S)),\iota_Y(\pi_2(S))).
\end{eqnarray*}

Since $S$ is a correspondence between $\pi_1(S)$ and $\pi_2(S)$ and $S\subset R$ then,
\begin{eqnarray*}
\dgh(\iota_X(\pi_1(S)),\iota_Y(\pi_2(S)))&\leq &\frac{1}{2} \dis(S)\\
	&=& \frac{1}{2}\max_{(x,y),(x',y')\in S} |d_X(x,x')-d_Y(y,y')|\\
	&\leq& \frac{1}{2}\max_{(x,y),(x',y')\in R} |d_X(x,x')-d_Y(y,y')|\\
    &=& \dis(R)\\
    &=& \dgh(X,Y).
\end{eqnarray*}

It follows that for all $S\subset R$, 
$$ |(\nu_n\circ \cset_n\circ\iota_X)(\pi_1(S))-(\nu_n\circ \cset_n\circ \iota_Y)(\pi_2(S))|\leq 2L\cdot \dgh(X,Y).$$

 We conclude that $d_{\F}((X,\Phi^{\nu_n}_X),(Y,\Phi^{\nu_n}_Y))\leq 2L\cdot\dgh(X,Y).$ The conclusion now follows directly from  Theorem \ref{thm:tripods}.\end{proof}

\subsection{Filtrations based on ultrametricity}

We now define a filtration based on ultrametricity, i.e. based on how close a simplex (endowed with the restriction of the metric) is from being an ultrametric space \cite{phylo}.

We will define the \emph{ultrametric deviation function} on a metric space $(X,d_X)$ to be:
$$ \ud_X(x_1,x_2,x_3):= d_{X}(x_1,x_3)- \max\{d_X(x_1,x_2),d_X(x_2,x_3)\}, \hspace{1cm} \forall x_1,x_2,x_3\in X. $$

This function  measures how much a triangle $(x_1,x_2,x_3)$ fails to be an ultrametric triangle. Indeed, on ultrametric spaces all triangles are isosceles and the two largest sides are equal \cite{phylo}. Then, we can define the ultrametricity \cite{tree-nips} of a metric $(X,d_X)$ space to be:

$$ \ult(X):= \max_{x_1,x_2,x_3\in X} \ud_X(x_1,x_2,x_3).$$

We  define the filtration functor $\Phi^{\ult}:\M\rightarrow \mathcal{F}$ as follows: for any $X\in\M$:
$$\Phi^{\ult}_X(\sigma) := \ult(\iota_X(\sigma) ),\hspace{1cm} \forall \, \sigma \subset X.$$

\begin{prop}\label{ultWD}
$\Phi^{\ult}$ is well defined. 
\end{prop}

\begin{proof}
Let $X\in \M$ and $\tau \subset \sigma \subset X$. Then,
\begin{eqnarray*}
	\Phi^{\ult}_X (\tau)&=& \max_{x_1,x_2,x_3\in \tau} \ud_X(x_1,x_2,x_3)\\
    &\leq & \max_{x_1,x_2,x_3\in \sigma} \ud_X(x_1,x_2,x_3) \\
    &=& \Phi^{\ult}_X(\sigma).
\end{eqnarray*}
Thus, any subsimplex $\tau$ of $\sigma\subset X$ has an earlier time of arrival than $\sigma$.

\end{proof}

\begin{rem} 
$\Phi^{\ult}$ satisfies the monotonicity condition that allows it to generate a filtrated simplicial complex when it is applied to data, and moreover it satisfies the functorial inequality stated in the definition of filtration functors when restricted to the category of finite metric spaces with isometric maps. However, $\Phi^{\ult}$ does not satisfy the functorial inequality when considering the category of finite metric spaces with 1-Lipschitz maps. This is proved below. 

\begin{proof}
Let $(\{p_1,p_2,p_3\},d_P)$ be the metric space of the equilateral triangle of side 2. Let $(\{s_1,s_2,s_3\}, d_R)$ the space consisting of two points of the equilateral triangle and the midpoint of a different side:

\begin{figure}[H]
\begin{tikzpicture}[scale=0.6,every node/.style={scale=0.6}]
\draw[shift={(-2,4*0.8660)}] (0,-0.4) node[] {$(P,d_P)$};
\filldraw[blue] (-2,0) circle (2pt) node[left] {$p_1$};
\filldraw[blue] (2,0) circle (2pt) node[right] {$p_2$};
\filldraw[blue] (0,4*0.8660) circle (2pt) node[above] {$p_3$};
\draw[blue] (-2,0) -- (2,0);
\draw[blue] (2,0) -- (0,4*0.8660);
\draw[blue] (0,4*0.8660) -- (-2,0);
\draw[shift={(0,0)}] (0,-0.4) node[] {2};
\draw[shift={(-1,2*0.8660)}] (-0.3,0.2) node[] {2};
\draw[shift={(1,2*0.8660)}] (0.3,0.2) node[] {2};

\draw[shift={(7-2,4*0.8660)}] (0,-0.4) node[] {$(S,d_S)$};
\filldraw[red] (7-2,0) circle (2pt) node[left] {$r_1$};
\filldraw[red] (7,0) circle (2pt) node[right] {$r_2$};
\filldraw[red] (7,4*0.8660) circle (2pt) node[above] {$r_3$};
\draw[red] (7-2,0) -- (7,0);
\draw[red] (7,0) -- (7,4*0.8660);
\draw[red] (7,4*0.8660) -- (7-2,0);
\draw[shift={(7-1,0)}] (0,-0.4) node[] {1};
\draw[shift={(7-1,2*0.8660)}] (-0.3,0.2) node[] {2};
\draw[shift={(7,2*0.8660)}] (0.3,0) node[] {$\sqrt{3}$};
\end{tikzpicture}
\caption{A visual representation of the spaces $P$ and $S$. It is clear that $P$ is ultrametric and that S is not.}
\end{figure}
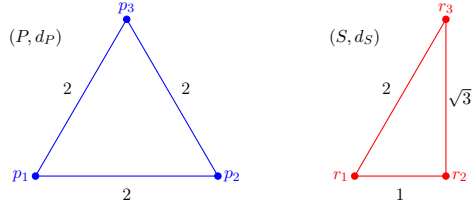
We can define a map $f: P\rightarrow S$ given by $f(p_i)=s_i$ for all $i\in \{1,2,3\}$. Clearly this map satisfies:
$$d_P(p_i,p_j)\geq d_S(f(p_i),f(p_j)) \hspace{1cm} \forall \, p_i,p_j \in P$$
It is also clear that $\Phi^{\ult}_P(\{p_1,p_2,p_3\})= 0$ since this triangle is isosceles and, then, ultrametric. But,
$$\Phi^{\ult}_S(\{s_1,s_2,s_3\})= 2-\sqrt{3}> 0 = \Phi^{\ult}_P(\{p_1,p_2,p_3\}).$$
\end{proof}
\end{rem}

The Rips functor determines the time of arrival of simplexes by their size. This is not the case for $\Phi^{\ult}$. Huge simplexes with ultrametric structure arrive at time zero, and non-ultrametric small simplexes are assigned positive time of arrival.

Next we show that $\Phi^{\ult}$ satisfies both stability and locality properties.

\begin{prop}\label{prop-ult}
$\Phi^{\ult}$ is $3$-local and is induced  by the valuation $\nu_{\ult} \in \mathfrak{V}_3$ such that $$ \R^{3\times 3} \supset A \mapsto  \max_{a\in A} (a_{13}-\max\{a_{12}, a_{23} \}).$$
Furthermore, $\nu_{\ult}$ is $2$-stable.
\end{prop}

\begin{proof}[Proof of Proposition \ref{prop-ult}]
 Let $A,B\subset \R^{3\times 3}$. Let $\delta= \dH^{(\mathbb{R}^{3\times 3}, \ell_\infty)}( A, B )$. Let $a^0\in A$ be a matrix in $A$ such that $\nu_{\ult}(A)= a^0_{12}-\max\{a^0_{13},a^0_{23}\}$. There exists a matrix $b^0\in B$ such that $||a^0-b^0||_{\infty}\leq \delta$. Then,

\begin{eqnarray*}
	\nu_{\ult}(A)-\nu_{\ult}(B)&=& a^0_{13}-\max\{a^0_{12},a^0_{23}\}-\nu_{\ult}(B) \\
    &\leq & a^0_{13}-\max\{a^0_{12},a^0_{23}\}- \left(b^0_{13}-\max\{b^0_{12},b^0_{23}\}\right)\\
    &=& (a^0_{13}-b^0_{13})+ \left( \max\{a^0_{12},a^0_{23}\} - \max\{b^0_{12},b^0_{23}\}\right)\\
    &\leq & |a^0_{13}-b^0_{13}| + \max\{|a^0_{12} - b^0_{13}|,|a^0_{22} - b^0_{23}| \} \\
    & \leq & 2||a^0 - b^0||_{\infty}\\
    & \leq & 2 \delta.
\end{eqnarray*}
It follows that the valuation $\nu_{\ult}$ is 2-stable. To conclude, we observe that for all $X \in \M$ and $\sigma\subset X$,
\begin{eqnarray*}
	\Phi^{\ult}_X(\sigma)&=& \max_{x_1,x_2,x_3\in \sigma} d_X(x_1,x_2)- \max\{d_X(x_1,x_3), d_X(x_2,x_3)\} \\
    &=& \max_{a\in\cset_3(\iota_X(\sigma)) } \left(a_{13}-\max\{a_{12}, a_{23} \} \right) \\
    &=& \nu_{\ult} (\iota_X(\sigma)).
\end{eqnarray*}

\end{proof}

Let us denote the $k-$th dimensional diagram map induced by $\Phi^{\ult}$ as $\dgm^{\ult}_k$. Note that, from Theorem \ref{lstablevaluation}, $\Phi^{\ult}$ is 4-stable: for all $k\in\N$ and $X,Y\in\mathcal{M}$, $$ \dB\left(\dgm^{\ult}_{k}(X), \dgm^{\ult}_{k}(Y) \right)\leq 4\cdot \dgh(X,Y).$$

From this, $\Phi^{\ult}$ is an example of a filtration functor induced by a stable valuation that satisfies functoriality on a different category than both Vietoris-Rips and \v{C}ech filtration functors. There are other interesting features of this filtration that are essentially different from the intuition followed by the Vietoris-Rips or the \v{C}ech filtration functors:

\begin{rem} We make the following remarks:
\begin{enumerate}
\item For any two-point metric space $P=\{p_1,p_2\}$ with distance $d_P$ one has $\ult(P)=0$. Then, for any $X\in \M$, the whole $1$-dimensional skeleton of $X$ is added by the filtration $\Phi^{\ult}$ at time zero. This implies that $\dgm_0^{\ult}(X)=\{[0,\infty)\}$ for any $X\in\mathcal{M}.$

\item The time of appearance of a simplex $\sigma \subset X$ is given by a function of triplets of points. Thus, the underlying simplicial complex is the clique/flag complex of the complex built up to dimension $2$ at all time. Furthermore, since  the whole $1$-skeleton is added at time zero, the whole complex depends only on the birth time of $2$-dimensional simplexes.
\end{enumerate}
\end{rem}

\subsubsection{A family of ultrametricity based filtration functors}

We generalize the above notion of ultrametricity in order to generate a whole family of filtrations.

\defn{Let $k\in \N$ and $(X, d_X)\in \M$. We define the \emph{$k$-th ultrametricity of $X$} to be $$\ult_k(X):= \max_{x_1,...,x_k\in X} \left(d_X(x_1,x_k)- \max_{i} d_X(x_{i-1},x_i)\right). $$}

The $k-$th ultrametricity intends to measure how a space differs from being ultrametric on sets of size at most $k$. Notice that $\ult =\ult_3$. 

Using this definition we now define a new family of filtration functors.

\defn{ Let $k\in \N$. The \emph{$k-$th ultrametricity filtration functor} is the map $\Phi^{\ult_k}:\M\rightarrow \F$ given by,
$$\Phi^{\ult_k}_X(\sigma):= \ult_k(\iota_X(\sigma)), \hspace{1cm} \forall \, X\in \M , \, \sigma\subset X. $$
}

Of course $\Phi^{\ult_3}=\Phi^{\ult}$. As in Proposition \ref{prop-ult} we can prove locality and stability of these filtration functors.

\begin{prop}
For each $k\in\N$, $\Phi^{\ult_k}$ is well defined, $k$-local, and induced by the $2$-stable valuation $\nu_{\ult_k}\in\mathfrak{V}_k^2$ given by
$$ \R^{k\times k}\supset A \mapsto\nu_{\ult_k}(A):= \max_{a\in A} \left(a_{1k}- \max_{i} a_{i\, i+1}\right).$$
\end{prop}

We will denote the $k-$th ultrametricity filtration functor as $\Phi^{\ult_k}$ instead of $\Phi^{\nu_{\ult_k}}$. We now prove that this family of filtrations is not trivial.

\begin{prop}
In general, for $k>\ell \geq 3$, $\ult_k$ and $\ult_{\ell}$ are different functions such that $\ult_k \geq \ult_{\ell}$. Furthermore, $\Phi^{\ult_k}\neq \Phi^{\ult_{\ell} }.$
\end{prop}
\begin{proof}
	First let us prove the inequality. Let $k> \ell \geq 3$. Let $x_1,...,x_\ell\in X$ be a path of $\ell$ points in $X$. We can build a path of $k$ points by defining $x_i=x_k$ for $\ell <i \leq k$. From this, it follows that,
    $$ d_X(x_i,x_\ell)-\max_{1< i\leq \ell}\{d_X(x_{i-1},x_i)\}=  d_X(x_i,x_k)-\max_{1< i\leq k}\{d_X(x_{i-1},x_i)\}.$$
    This shows that any path of $\ell$ points can be considered as a degenerated path of $k$ points. Since the maximum on $\ult_{k}$ considers paths with at most $k$ points, and $\ult_\ell$ just paths of $\ell$ points, it follows that,
\begin{eqnarray*}
    \ult_{\ell}(X)&=& \max_{x_1,...,x_\ell\in X} d_X(x_i,x_\ell)-\max_{1< i\leq \ell}\{d_X(x_{i-1},x_i)\}\\
    &\leq & \max_{x_1,...,x_k\in X} d_X(x_i,x_k)-\max_{1< i\leq k}\{d_X(x_{i-1},x_i)\}\\
    &=&\ult_k(X). 
\end{eqnarray*} 
 To prove that they are different, we must exhibit a space in which they differ. Let $X_k= \{1,...,k\}$, endowed with the subspace metric induced by $\R$. We will prove that $\ult_k(X_k)=k-2$ and $\ult_{\ell}(X_k)<k-2$. We can see that,
	$$\ult_k(X)\geq |k-1|-\max\{|2-1|,...,|k-(k-1)|\}= k-2.$$
Now, if we choose any path such that $x_1,x_k$ are not $1$ and $k$ then, $|x_k-x_1|\leq k-2$ and then $|x_k-x_1|-\max\{|x_2-x_1|,...,|x_k-x_{k-1}|\}\leq k-2$. If $x_1=1$ and $x_k=k$, then for some $i\in \{2,...,k\}$, $|x_{i}-x_{i-1}|\geq 1$ and then,
$$ |x_k-x_1|-\max_{1< i\leq k}\{|x_{i}-x_{i-1}|\}\leq k-2.$$
It follows that $\ult_k(X_k)= k-2.$
Now, let us prove that $\ult_\ell(X_k)<k-2$. We will prove this by contradiction. Let assume that $\ult_\ell(X_k)\geq k-2$. Then, there is a path $x_1,...,x_\ell$ such that,
$$ |x_\ell-x_1|-\max_{1< i\leq \ell}\{|x_{i},x_{i-1}|\}\geq k-2. \hspace{1cm} (\ast)$$
This implies that $|x_\ell-x_1|\geq k-2.$ If $|x_\ell-x_1|= k-2$ then $\max_{1< i\leq \ell}\{|x_{i}-x_{i-1}|\}=0$. From this, it follows that $x_1=x_2=...=x_\ell$. This holds a contradiction since $x_1\neq x_\ell$.
Now, if $|x_\ell-x_1|>k-2$ then $|x_\ell-x_1|= k-1$ and, without loss of generality, $x_1= 1$ and $x_\ell= k$. From $(\ast)$, it follows that $\max_{1< i\leq \ell}\{|x_{i}-x_{i-1}|\}\leq 1$. Then,
	$$|x_i|\leq |x_1|+|x_2-x_1|+...+|x_i-x_{i-1}|\leq  i.$$
Hence, $x_{\ell-1}<\ell-1$ and, 
    $$|x_{\ell-1}-x_\ell|\geq  k-(\ell-1)\geq 2.$$ 
    We conclude that $\max_{2\leq i\leq \ell} |x_{i}-x_{i-1}|\geq 2$, from which $(\ast)$ does not hold.
    Since $\ult_k(X_k)\neq \ult_\ell(X_k)$, then $\Phi_{X_k}^{\ult_k}(X_k)\neq \Phi_{X_k}^{\ult_\ell}$, and this filtration functors are different.
\end{proof}

We will see computational experiments related to $\Phi^{\ult}$ in Section \ref{sec:exp}.

\subsection{The hyperbolicity functor}
 Given a metric space $X\in\M$, we can define the hyperbolicity deviation function \cite{phylo} to be 
 \begin{eqnarray*}
 	\hd_X(x_1,x_2,x_3,x_4)&:=& \frac{1}{2}\big(d_X(x_1,x_2)+d_X(x_3,x_4)\\
    & & \hspace{1cm}- \max\{d_X(x_1,x_3)+d_X(x_2,x_4),d_X(x_1,x_4)+d_X(x_2,x_3) \}\big),
 \end{eqnarray*}
for all $x_1,x_2,x_3,x_4\in X$. With this, we can define the hyperbolicity of a metric space, given by:
 $$ \hyp(X):= \max_{x_1,x_2,x_3,x_4\in X} \hd_X(x_1,x_2,x_3,x_4), \hspace{1cm} \forall \, X\in \M.$$
 
 This quantity measures how far a metric space is from being a tree-shaped space. From these, we define a filtration functor $\Phi^{\hyp}: \M\rightarrow \F$ to be, 
 $$ \Phi^{\hyp}_X(\sigma):= \hyp(\iota_X(\sigma)) \hspace{1cm} \forall \, X\in \M, \, \sigma\subset X.$$
 
 Analogously to the proof of Proposition \ref{prop-ult}, we can prove stability and locality of the hyperbolicity filtration functor.
 
\begin{prop}\label{hyp_prop}
 The map $\Phi^{\hyp}$ is a well defined $4-$local filtration functor, generated by the valuation:
 $$\nu_{\hyp}(A)= \frac{1}{2}\left(\max_{a\in A} \left(a_{12}+a_{34}-\max\{a_{13}+a_{24},a_{14}+a_{23}\} \right)\right). $$
 Furthermore, this valuation is $2-$stable.
\end{prop}

As we mentioned before about the ultrametricity filtration, it is worth noticing that this filtration functor adds all the $3-$dimensional skeleton at time zero, since the hyperbolicity of a 3-simplex is zero. Then, the first interesting persistence diagram to compute is the fourth-dimensional.

Although the computational complexity of this filtration is polynomial, it is hard to give explicitly since it relies on at least $\binom{n}{4}$ computations. This quantitiy scales as $O(n^4)$, which makes this filtration computationally difficult. Nevertheless, it is still a natural filtration to use when one is concerned with the ``treeness'' of a space.

\section{Basepoint Filtration Functors}

A different way of broadening our understanding of filtration functor is to generate new concepts of what a filtration is. It is possible to exploit further information about our spaces through filtrations that depend on choosing a basepoint that provides perspective. 

\defn{For any $X\in \M$, a   \emph{basepoint family of filtrations on $X$} is any collection 
$$\{\psi_X(x_0):\pow (X) \rightarrow \raggedbottom \F \}_{x_0\in X},$$ where for all $x_0\in X$, $\psi_X(x_0)$ is a filtration on $X$. We call $X$ the base space of the basepoint family of filtrations.\\
}

With a basepoint family of filtration, instead of choosing a single filtration for the space, for each point on the base space we have a different way of filtering it. We can also define a joint rule that, for each space in $X\in \M$, for each point $x_0$ in $X$, we assign a filtration of $X$ based on $x_0$.

\defn{A \emph{basepoint filtration functor} is a map $\Psi:\M \rightarrow \pow(\F)$, such that $\Psi$ takes any $X\in \M$ to a basepoint family of  filtrations $\Psi_X= \{\psi_X(x_0): \pow(X) \rightarrow \R \}_{x_0\in X}$.} 
Furthermore, for the property of being a functor we ask that if $(X,d_X),(Y,d_Y)\in \M$ and $\phi: X\rightarrow Y$ is such that $\forall x,x'\in X$, $d_X(x,x')\ge d_Y(\phi(x),\phi(x'))$, and if we let $\Psi_{X,x_0}=\psi_X(x_0)$, then for all $\sigma\subseteq X$ and $x_0\in X$, one has $\Psi_{X,x_0}(\sigma)\geq \Psi_{Y,\phi(x_0)}(\phi(\sigma)).$\\

Note that any filtration as we have previously defined gives rise to a basepoint filtration.

\begin{ex}\label{ex:constantbff}[Constant basepoint filtration functors]
Let $\Phi:\M \rightarrow \F$ be any filtration functor. We define the constant basepoint filtration functor $\Psi^{\mathrm{const}}$ to be
	$$\Psi^\mathrm{const}_X(x_0)= \Phi_X, \hspace{1cm} \forall \, X\in \M, \, x_0\in X.$$
For example, for all $x_0\in X$, one could define $\Psi_X(x_0)$ to be the Vietoris-Rips filtration. Then this basepoint filtration functor would be equivalent to the general Vietoris-Rips filtration.
\end{ex}

As before, we have to impose stability conditions on this filtrations so that there is a relation not only among spaces, but also among filtrations on same spaces with different basepoints. First, we define the cost function of a filtration functor, and then we use it to define a notion of stability.

\defn{\label{def:cost}
Given a basepoint filtration functor $\Psi$, $X,Y\in\mathcal{M}$ and $k\geq 0$ consider the function $\mathcal{C}_{\Psi,k}:X\times Y\rightarrow \R_+$ given by 
$$(x,y)\mapsto\dB\big(\dgm_k^{\Psi_X(x)}(X),\dgm_k^{\Psi_Y(y)}(Y)\big).$$ This function is called the \emph{k-dimensional cost function induced by $\Psi$.} }

\defn{
	Let $L>0$. Given a basepoint filtration functor $\Psi$, we say it is an \emph{L-}stable basepoint filtration functor if, for all $k\geq 0$,
    
    $$ \min_{R\in\mathcal{R}(X,Y)}\max_{(x_0,y_0)\in R}\mathcal{C}_{\Psi,k}(x_0,y_0)\leq L\cdot\dgh(X,Y), $$
for all $X,Y\in \M$.
}

Why are we not saying that a basepoint filtration functor is stable if for \emph{all pairs} $(x,y)\in X\times Y$ the above inequality holds? The issue is that such a restriction is too strong---it is possible that an arbitrary pair of basepoints will be incomparable to one another, in a sense determined by the user. For example, when comparing two images of a cat, it may not make sense to choose the tip of the tail as one basepoint and a point at the center of its body as another (see Section \ref{sec:exp}). 

Instead of asking for any pairwise comparison to be well-behaved, we define a functor to be stable if there is a set $R\subset X\times Y$ such that the comparisons of points in it are meaningful, and that the projections of the pairs in $R$ sufficiently cover both $X$ and $Y$.

\subsection{Local Basepoint Filtration Functors}
Let $n\in \N$. For all $X\in \M$, $x\in X$ and $\sigma \subset X$, we define the \emph{basepoint n-th curvature set of $\sigma$} to be,

$$ \cset_n(x_0,\sigma) :=  \big(D^{(n+1)}_X(x_0,\sigma^n)\big) \subset \R^{(n+1)\times (n+1)}.$$
Consider the projection $\pi_n:\R^{(n+1)\times (n+1)}\rightarrow \R^{n\times n}$ given by 
$$\big(a_{ij}\big)_{i,j=1}^{n+1}\mapsto \big(a_{ij}\big)_{i,j=2}^{n+1}.$$
Notice that for every $x_0\in X$ one has $\pi_n(\cset_n(x_0,\sigma)) = \cset_n(\iota_X(\sigma)).$

Given a natural number $n$ and any stable valuation $\nu_{n+1}$ one can construct the  basepoint filtration functor $\Psi^{\nu_{n+1}}$ defined as follows:  $X\in\mathcal{M}$, 

$$\Psi^{\nu_{n+1}}_X:X\rightarrow \F,$$
$$ x\mapsto \Psi^{\nu_{n+1}}_X(x),$$

where the time of arrival of a simplex $\sigma\subset X$ is given by $$\Psi^{\nu_{n+1}}_X(x)(\sigma)= \nu_{n+1}(\cset(x,\sigma) ).$$

The family $\Psi^{\nu_{n+1}}$ is a basepoint family of filtration functors induced by $\nu_{n+1}.$ This is a possible path through which we can generate basepoint filtration functors. But it has a disadvantage. Although we are considering information about the basepoint and the simplex, we are not taking account of the particular position of the basepoint in the space. This could be measured by other global quantities that cannot be computed using only the basepoint curvature sets, which are intrinsically local. For this, we define first the notion of functions that describe the point in the whole space.

For this, let $\M_\ell$ be the collection of all triplets $(X,d_X,f_X)$ where $(X,d_X)$ is a finite metric space and $f_X: X\rightarrow \R^{\ell} $ is a function. We will think of a map from $\M$ to $\M_\ell$ to be a relation that, for each $X\in \M$ and a point $x_0\in X$, assigns a set of quantities that describe the position of the point with respect to the space.

\begin{defn}
A \emph{point descriptor} $\rho$ is a functor from $\M$ to $\M_\ell$,
$$\rho:\M\rightarrow \M_\ell$$
$$X\in \M \mapsto \rho_X:X\rightarrow \R^\ell $$
such that there exists a constant $K>0$ such that for all $X,Y\in \M$ and $R\subset X\times Y$ correspondence,
$$ \max_{(x_0,y_0)\in R}|\rho_X(x_0)-\rho_Y(y_0)|\leq K \cdot \dis(R).$$ 
If this condition is satisfied by $K\geq 0$ we say this point descriptor is \emph{$K-$stable}.
\end{defn}

Once we defined the notion of point descriptors, we need a way to match it with assigning a time of arrival for each simple, such that those satisfy the monotonicity condition. This will be fulfilled by the definition of adjusted valuations.

\begin{defn}[Adjusted valuation]
For some $\ell\in \N$, an \emph{adjusted valuation} is a map $$\nu_{n,\ell}:\pow(\R^{n\times n})\times \R^\ell \rightarrow \R_+$$ with an altered version of monotonicity:  for any fixed $v\in \R^\ell$, $\nu_{n,\ell}(A\times v)\geq \nu_{n,\ell}(B\times v)$ for all $B\subset A\in \pow(\R^{n\times n})$.
We will say that an adjusted valuation $\nu_{n,\ell}$ is \emph{$L$-stable} if for all $A,B\in\pow(\R^{n\times n})$ and $v,w\in \R^\ell$,
$$ |\nu_{n,\ell}(A,v_1)- \nu_{n,\ell}(B,w)|\leq L\cdot\max\{\dH(A,B), ||v-w||_\infty\}.$$
\end{defn}

We can combine the notion of a point descriptor and an adjusted valuation to generate interesting basepoint filtration functors. These will be more general than the straightforward filtrations we obtain by just applying valuations to the basepoint curvature set. And they can carry and exploit more information without losing their computability.

\begin{defn}[Local basepoint filtration functor]
A basepoint filtration functor $\Psi$ is a \emph{$(n,\ell)-$local} if there exists a point descriptor $\rho$ with image in $\R^\ell$, an adjusted valuation $\nu_{n,\ell}:\pow(\R^{(n+1)\times (n+1)})\times \R^\ell \rightarrow \R_+$ such that for any $X\in \M$ and $x_0\in X$:
$$\Psi_X(x_0)(\sigma )=\nu_{n,\ell}(\cset_n(x_0,\sigma) , \rho(x_0)),\hspace{1cm}\forall \, \sigma \subset X.$$
\end{defn}

\subsection{Stability Results for Local Basepoint Filtrations}
\begin{thm}
\label{thm:kn-local}
Let $\Psi$ be a local basepoint filtration functor. Let $\nu_{n,\ell}$ be a $L-$stable adjusted valuation and $\rho$ a $K-$stable point descriptor with image in $\R^{\ell}$. Lets assume that $\Psi$ is generated by $\nu_{n,\ell}$ and $\rho$. Then, for all $k\geq 0$,
 $$ \min_{R\in \mathcal{R}(X,Y)}\max_{(x_0,y_0)\in R}\mathcal{C}_{\Psi,k}(x_0,y_0)\leq 2L\cdot\max\{1,K\}\cdot \dgh(X,Y).$$
\end{thm}
\begin{proof}[Proof of Theorem \ref{thm:kn-local}]
Let $X,Y\in \M$ and $R_0\in \mathcal{R}(X,Y)$ such that $\frac{1}{2}\dis(R_0)= \dgh(X,Y)$. Now, notice that if $\pi_X,\pi_Y$ are the canonical projections from $R_0$ to $X$ and $Y$ respectively, $(R_0,\pi_X,\pi_Y)$ is a triplet with a set and surjective maps from it to $X$ and $Y$, respectively.

Now, let $\sigma\subset R_0$, and $(x,y)\in R_0$. Notice that $(\sigma,\pi_X|_\sigma,\pi_Y|_\sigma)$ is a triplet of a set and two surjective maps to $\pi_X(\sigma)$ and $\pi_Y(\sigma).$ Let $\sigma_X=\pi_X(\sigma)$ and $\sigma_Y= \pi_Y(\sigma)$ be the metric spaces generated by the projections union the basepoint. We observe that,
\begin{eqnarray*}
|\Psi_X(x)(\pi_X(\sigma))-\Psi_Y(y)(\pi_Y(\sigma))|&=& |\nu_{n,\ell}(\cset_n(x,\sigma_X),\rho_X(x)) - \nu_{n,\ell}(\cset_n(y,\sigma_Y),\rho_Y(y)) |\\
&\leq & L \cdot \max\{\dH(\cset_n(x,\sigma_X ),\cset_n(y,\sigma_Y )) ,\\
& & \hspace{1cm} ||\rho_X(x)-\rho_Y(y)||_\infty\}.\\
\end{eqnarray*}
 From the definition of stability of point descriptors, we see that,
$$ ||\rho_X(x)-\rho_Y(y)||_\infty \leq K\cdot\dis(R_0)= 2K\cdot\dgh(X,Y).$$

Now, let $\alpha= D^{(n+1)}_X(x_0,x_1,...,x_n)\in \cset_n(x,\sigma_X)$. There are $y_1,...,y_n$ such that $(x_i,y_i)\in \sigma$ for all $i\in \{1,...,n\}$. Then, if $\beta= D^{(n+1)}_Y(y_0,y_1,...,y_n),$
\begin{eqnarray*}
||\alpha-\beta||_\infty&= &\max_{0\leq i,j\leq n}\left\{|d_X(x_i,x_j)-d_Y(y_i,y_j)| \right\}\\
&\leq & \dis(R_0)\\
&=& 2\cdot\dgh(X,Y).
\end{eqnarray*}
Then, 
\begin{eqnarray*}
|\Psi_X(x)(\pi_X(\sigma))-\Psi_Y(y)(\pi_Y(\sigma))|&\leq &L\cdot\max\{2\cdot\dgh(X,Y), 2K\cdot\dgh(X,Y)\}\\
&=& 2L\cdot \max\{1,K\}\cdot \dgh(X,Y).
\end{eqnarray*}
It follows that for all $(x_0,y_0)\in R_0$,
$$d_{\mathcal{F}}((X, \Psi_X(x_0)),(X, \Psi_Y(y_0)) )\leq  2L\cdot\max\{1,K\}\cdot \dgh(X,Y) .$$
From this, we conclude that,
\begin{eqnarray*}
	\min_{R\in\mathcal{R}(X,Y)}\max_{(x_0,y_0)\in R} d_{\mathcal{F}}((X, \Psi_X(x_0)),(X, \Psi_Y(y_0)) )&\leq &\max_{(x_0,y_0)\in R_0} d_{\mathcal{F}}((X, \Psi_X(x_0)),(X, \Psi_Y(y_0)) )\\ 
    &\leq&  2L\cdot\max\{1,K\}\cdot \dgh(X,Y).
\end{eqnarray*}

From Theorem \ref{thm:tripods}, we conclude that for any $k\geq 0$,
	$$\min_{R\in\mathcal{R}(X,Y)}\max_{(x_0,y_0)\in R}\mathcal{C}_{\Psi,k}(x_0,y_0)\leq 2L\cdot\max\{2,K\}\cdot \dgh(X,Y).$$
    
\end{proof}

This theorem is the most general stability theorem we prove in this work. It is a generalization of Theorem \ref{lstablevaluation}, considering Example \ref{ex:constantbff}.  

We know that the local filtration functors are well behaved when generated by stable adjusted valuations and point descriptors. Now, we must consider a different type of stability. It should be true that changing the basepoint transforms the diagrams in a continuous way. This would align to the idea that the filtration depends on the perspective of the point; if the perspective changes by a small distance, the induced diagrams should incur small changes.

\begin{prop}\label{Stab_loc}
Let $X\in \M$ and $x,x'\in X$. Let $\Psi$ be a local filtration functor. Let  $\nu_{n,\ell}$ be an $L-$stable adjusted valuation and $\rho$ a $K-$stable point descriptor with image in $\R^{\ell}$. Let us assume that $\nu_{n,\ell}$ and $\rho$ generate the basepoint filtration functor $\Psi$. Then, for all $k\in \N,$
$$\dB(\dgm^{\Psi_X(x)}_k(X),\dgm^{\Psi_X(x')}_k(X)) \leq L\cdot\max\{1,K\}\cdot d_X(x,x'). $$
\end{prop}
\begin{proof}

Recalling the definition of $d_{\mathcal{F}}$ and Theorem \ref{thm:tripods}, it is sufficient to show that there is a triplet $(Z,\pi_X,\pi_Y)$ of a set and surjective maps from it to $X$ and $Y$ respectively such that, 
$$\max_{\sigma\subset Z}|\Phi_X(x)(\pi_X(\sigma))-\Phi_X(x')(\pi_X(\sigma))|\leq L\cdot\max\{1,K\}\cdot d_X(x,x'). $$

Let $X\in \M$ and $x,x'\in X$. Let we see that $(X,\id_X,\id_X)$ is a triplet of a set and two surjections to the set $X$. Let $\sigma\subset X$. Then, 
\begin{eqnarray*}
	|\Psi_X(x)(\id_X(\sigma))- \Psi_X(x')(\id_X(\sigma)) |&=& |\Psi_X(x)(\sigma)- \Psi_X(x')(\sigma) |\\
	&=& |\nu_{n,l}(\cset(x,\sigma),\rho_X(x) )-\nu_{n,l}(\cset(x,\sigma),\rho_X(x') )|\\
    &\leq & L\cdot\max\{\dH(\cset(x,\sigma),\cset_n(x',\sigma)),\\
    & & \hspace{1cm}|| \rho_X(x)-\rho_X(x')||_\infty \}.
\end{eqnarray*}
First, if we choose the correspondence $R=\{(x,x)\}_{x\in X}\cup\{(x,x')\}$ then it follows from the definition that,
\begin{eqnarray*}
\dis(R)&=&\max_{(x_1,x_2),(x_1',x_2')\in R} |d_X(x_1,x_1')-d_X(x_2,x_2')|\\
	&=& d_X(x,x').
\end{eqnarray*}

Now, let $\alpha\in \cset_n(x,\sigma).$ There exists $x_1,...,x_n\in \sigma$ such that $\alpha= D^{(n+1)}_X(x,x_1,...,x_n)$. It follows that $\beta= D^{(n)}_X(x',x_1,...,x_n)$ is an element of $\cset_n(x',\sigma)$ and,

\begin{eqnarray*}
||\alpha-\beta||_\infty&=& \max_{1\leq i\leq n}|d_X(x,x_i)-d_X(x',x_i) |\\
&\leq & d_X(x,x').
\end{eqnarray*} 
Similarly, we can prove that if we choose $\beta\in \cset_n(x',\sigma)$, there is an element $\alpha\in \cset_n(x',\sigma)$ such that $||\alpha-\beta||_\infty\leq d_X(x,x')$. This implies that $\dH(\cset_n(x,\sigma),\cset_n(x,\sigma))\leq d_X(x,x').$ Then,
\begin{eqnarray*}
|\Psi_X(x)(\sigma)- \Psi_X(x')(\sigma)|&\leq & L\cdot\max\{ d_X(x,x'),K\cdot d_X(x,x')\}\\
&=& L\cdot\max\{1,K\}\cdot d_X(x,x'). 
\end{eqnarray*}
\end{proof}

\subsection{Eccentricity Filtration}

Given a compact metric space $(X,d_X)$, we recall the \emph{eccentricity function} (see \cite{dgh-props}) $\ecc_X:X\rightarrow \R_+$ to be $x\mapsto \max_{x'\in X}d_X(x,x').$

\begin{ex}[Eccentricity basepoint family]\label{eccfiltration}
A more interesting example of a basepoint filtration functor is $\Psi^\ecc$
such that for all $X\in \M$,
$$\Psi^\ecc_X = \{\psi_X^\ecc(x_0):\pow(X)\rightarrow \R\}_{x_0\in X},$$ 
where for $x_0\in X$ and $\sigma \subset X$, $$\psi_X^\ecc(x_0)(\sigma) := \max \left\{\diam( \iota_X(\sigma) ),\frac{1}{2}\left(\ecc_X(x_0)-\min_{x'\in \sigma}d_X(x_0,x')\right)\right\}.$$
 \end{ex}

\begin{lemma}
$\Psi^\ecc$ is well defined.
\end{lemma}
\begin{proof}
Let $X\in \M$ and $x_0\in X$. We need to then show that $\psi^{\ecc}_X(x_0)$ is a filtration on $X$. By definition, we already have that $\psi^\ecc_X(x_0)$ is a map from $\pow(X)\rightarrow \R_+$, so we only need to show that it satisfies the monotonicity condition. Let $\tau \subset \sigma \subset \pow(x)$. Then we have $\diam(\tau)\leq \diam(\sigma)$, and we have that $\min\limits_{x'\in\sigma}d_X(x_0,x') \leq \min\limits_{x'\in\tau}d_X(x_0,x')$ since $\tau\subset \sigma$. This implies $$\frac{1}{2}\left(\ecc_X(x_0)-\min\limits_{x'\in \sigma}d_X(x_0,x')\right)\geq \frac{1}{2}\left(\ecc_X(x_0)-\min\limits_{x'\in \tau}d_X(x_0,x')\right).$$ Putting these two together gives that $\psi^{\ecc}_X(x_0)(\tau)\leq \psi^{\ecc}_X(x_0)(\sigma)$, so the monotonicity condition holds, and $\psi^{\ecc}_X(x_0)$ is a filtration on $X$.
\end{proof}

We see that $\Psi^\ecc$ is a basepoint filtration functor. For this functor, $\ecc_X(x)$ is the point descriptor. Then $\forall A\in \cset_n(x_0,X)$, the adjusted valuation is given by 
$$\nu_{n+1,1}(A\times v ):= \max\left\{\max(a_{ij})_{2\leq i,j\leq n+1}, \frac{1}{2}\left(v-\min(a_{1j})_{2\leq j\leq n+1}\right)\right\}.$$

To prove stability of the eccentricity filtration, we just need to prove that the adjusted valuation and $\ecc$ are stable.
\begin{lemma}[\cite{dgh-props}]\label{lemma:easy}
Let $X,Y\in \mathcal{M}$ and let $R$ be any correspondence between $X$ and $Y$. Then,
\begin{enumerate}
\item $|\diam(X)-\diam(Y)|\leq \dis(R),$
\item For all $(x,y)\in R,$ $|\ecc_X(x)-\ecc_Y(y)|\leq \dis(R).$

\end{enumerate}
\end{lemma}

\section{Constructions of infinite families of stable valuations}
In this section we construct two different infinite families of stable valuations and  also provide some characterization results.

\subsection{$\max$-induced families valuations}

One general method through which valuations can be generated is to define a function $f:\R^{n\times n } \rightarrow \R$ and define the max-induced valuation to be,
$$\nu^f_n(A):= \max_{\alpha\in A} f(\alpha),\hspace{1cm} \forall \, A\in \pow(\R^{n\times n}). $$
One condition that is sufficient for this valuations to be stable is for $f$ to be a $L$-Lipschitz function. In this case, the valuation will be an $L-$stable valuation.

\begin{prop}
	Let $f:\R^{n\times n} \rightarrow \R$ be a $L-$Lipschitz function. Then the max-induced valuation $\nu^f_n$ is $L-$stable.
\end{prop}
\begin{proof}
	Let $A,B\in \pow(\R^{n \times n})$ and $\delta= \dH(A,B)$. Let $\alpha_0\in A$ such that $f(\alpha_0)= \nu^f_n(A)$. There exists $\beta_0 \in B$ such that $||\alpha_0-\beta_0||_\infty\leq \delta$. From the Lipschitz continuity,
    $$\nu^f_n(A)-\nu^f_n(B)\leq f(\alpha_0)- f(\beta_0)\leq L\,\delta. $$
	Analogously it follows that,
    $$|\nu^f_n(A)-\nu^f_n(B)|\leq L\cdot \delta. $$
\end{proof}

The Vietoris-Rips, Ultrametricity, and Hyperbolicity filtration functors all belong to this family, since they were induced respectively by the following functions: 
\begin{eqnarray*}
f^{\textrm{Rips}}(\alpha)= \alpha_{12}, &\hspace{1cm}& \forall \alpha \in \R^{2\times 2},\\
f^{\ult}(\beta)= \beta_{13}-\max\{\beta_{12},\beta_{23}\},&\hspace{1cm}& \forall \beta \in \R^{3\times 3},\\
f^{\hyp}(\gamma)= \gamma_{12}+\gamma_{34}-\max\{\gamma_{13}+\gamma_{24},\gamma_{14}+\gamma_{23}\},&\hspace{1cm}& \forall \gamma\in\R^{4\times 4}.
\end{eqnarray*}

\subsection{Other families of valuations}
We have a generating method of valuations through composition of a maximum and a Lipschitz continuous function with respect to both $\ell_\infty$ metrics. The Rips, ultrametricity, and hyperbolicity valuations follow this structure. This  the question: Is there any stable valuation that does not behave this way? In this section we further break down the Rips valuation to gain a method of constructing families of valuations different from this pattern.
\subsubsection{1-point valuation}
 
Using a slightly different procedure we can create a valuation whose corresponding filtration is equivalent up to some scalar to the Rips filtration. 

\begin{defn}
Let $\omega_{n,1}(A)$ be the \emph{1-point valuation} for $A\in \pow(\R^{n\times n})$ defined as,
$$\omega_{n,1}(A)= \inf\limits_{p\in \R^{n\times n}}\max\limits_{\alpha\in A}||\alpha-p||_{\infty}.$$
\end{defn}

When we analyze the time of arrival of simplexes through curvature sets and valuations, we are just evaluating the valuation in a particular family of sets of matrices. For any $X\in \M$, $\cset_n(X)$ has no matrix with negative entries. It is also true that for any curvature set, the zero matrix will be in the set, by taking any constant $n-$tuple.

\begin{defn}
We call a finite set of matrices $A\in \pow(\R^{n\times n})$ almost metric if $0\in A$ and for all $\alpha\in A$, all entries of $\alpha$ are nonnegative.
\end{defn}

We now prove that when restricted to almost metric sets, $\omega_{n,1}$ agrees with the $\max$-valuation induced by the function $\alpha\mapsto \frac{1}{2}||\alpha||_\infty$.

\begin{prop}
Let $A\in \pow(\R^{n\times n})$ an almost metric set. Then,
$$\omega_{n,1}(A)=\frac{1}{2}\max\limits_{\alpha\in A}||\alpha||_{\infty}.$$
\end{prop}

\begin{proof}
Let $\alpha_0 \in A$ such that, 
$$||\alpha_0||_{\infty}= \max_{\alpha\in A} ||\alpha||_{\infty}.$$
Let $\varepsilon=\frac{||\alpha_0||_{\infty}}{2}$ and $p_0\in \R^{n\times n}$ the matrix with all entries equal to $\varepsilon$.  Then, 
\begin{enumerate}
	\item It is clear that $||\alpha_0- p||_\infty=||0-p||_{\infty}=\varepsilon $.
    \item If $\alpha \in A$, since $ 0 \leq \alpha_{ij}\leq 2\varepsilon$ then,
    $$||\alpha-p||_\infty = \max_{i,j\in \{1,...,n\}} |\alpha_{ij}-\varepsilon|\leq \varepsilon. $$
\end{enumerate}
From this, it follows that $\omega_{n,1}(A)\leq \varepsilon.$ Lets prove that this is the minimum. Let $p\in \R^{n\times n}$ be any point. We know that,
$$||0-p||_\infty+ ||p-\alpha_0||\geq ||0-\alpha_0||_\infty= 2\varepsilon. $$
From this, it follows that, 
$$ \max\{||0-p||_\infty, ||p-\alpha_0||\}\geq \varepsilon.$$
And, so,
$$ \max_{\alpha\in A} ||\alpha-p||_{\infty}\geq \varepsilon.$$

\end{proof}

From this, it follows immediately that, although this valuation is not a max-induced defined valuation, it still generates the Vietoris-Rips filtration functor.

\begin{cor}
	If $X\in\M$ and $\sigma\subset X$, then,
    $$(\omega_{n,1}\circ\cset_n\circ\iota_X)(\sigma)= R_X(\sigma). $$
\end{cor}
\begin{proof}
Let $X\in \M$ and $\sigma \subset X$. We notice that the zero matrix is in $\cset_n(\iota_X(\sigma))$, and all matrix entries of elements of $\cset_n(\iota_X(\sigma))$ are nonnegative. From the fact that all elements of $\cset_n(\iota_X(\sigma))$ have nonnegative entries bounded by $\diam(\iota_X(\sigma))$,
$$ \max_{\alpha\in \cset_n(\iota_X(\sigma))} ||\alpha||_\infty= \diam(\iota_X(\sigma)).$$
It follows that,
$$\Phi^{\omega_{n,1}}_X(\sigma)= \omega_{n,1}(\cset_n(\iota_X(\sigma)))=\frac{1}{2} \diam(\iota_X(\sigma))= \frac{1}{2}R_X(\sigma).$$
\end{proof}

\subsubsection{$k$-point valuations}
We know explore an extention of the filtration functor $\omega_{n,1}$. To generalize this, instead of infimizing on $p\in\R^{n\times n}$, we infimize over $\{p_1,p_2,\dots,p_k\}\subset\R^{n\times n}$.

\begin{defn} 
Let the \emph{$k$-point $n$-valuation} for $A\in \pow(\R^{n\times n})$ be 

 $$\omega_{n,k}(A) := \inf_{\{p_1,...,p_k\}\subset\R^{n\times n} }\max_{\alpha\in A} \min_{1\leq i\leq k}\|\alpha-p_i\|.$$
\end{defn}

It is worth stressing that $\{p_1,...,p_k\}$ is a way of denoting a set generated by choosing $k$ points of the space, without constraining them to be different.

The defined $\omega_{n,k}$ measure how well a set $A\in \pow(\R^{n\times n})$ can be approximated by a set of size $k\in \N$.

We can prove that $\omega_{n,k}$ is a well defined valuation that enjoys 1-stability.

\prop{For each $n,k\in \N$, $\omega_{n,k}:\pow(\R^{n\times n})\rightarrow \R$ is a 1-stable valuation.}
\begin{proof}
	First, let us prove that it is a valuation: let $A\subset B$ be finite subsets of $\R^{n\times n}$ and let $\{p_1,...,p_k\}\subset \R^{n\times n}$. It is clear, since $A\subset B$, that,
  $$ \max_{\alpha \in A} \min_{1\leq i\leq k} ||\alpha -p_i||_{\infty}\leq \max_{\beta \in B} \min_{1\leq i\leq k} ||\beta -p_i||_{\infty}. $$
 Then, infimizing over all $\{p_1,...,p_k\}\subset \R^{n\times n}$,
  $$ \omega_{n,k}(A)= \inf_{\{p_1,...,p_k\} \subset \R^{n \times n } }\max_{\alpha \in A} \min_{1\leq i\leq k} ||\alpha -p_i||_{\infty}\leq \inf_{\{p_1,...,p_k\} \subset \R^{n \times n } }\max_{\beta \in B} \min_{1\leq i\leq k} ||\beta -p_i||_{\infty}= \omega_{n,k}(B). $$
  Then, indeed $\omega_{n,k}$ is monotonic. 
  Now, we prove that $\omega_{n,k}$ is 1-stable: let $A,B\subset \R^{n\times n}$ finite subsets, $\delta= \dH(A,B)$ and $\varepsilon>0$. Now, let $\{p^A_1,...,p^A_k\}\subset \R^{n\times n}$ a set of points such that 
  $$\omega_{n,k}(A)\leq\max_{\alpha \in A}\min_{1\leq i\leq k}||\alpha-p^A_i||_\infty < \omega_{n,k}(A)+\varepsilon.$$ 
  We see that,
\begin{eqnarray*}
	\omega_{n,k}(B)-\omega_{n,k}(A)&\leq & \max_{\beta \in B}\min_{1\leq i\leq k} ||\beta-p^A_i||_\infty -\omega_{n,k}(A)\\
    &\leq & \max_{\alpha \in A}\max_{\beta\in B}\left(||\beta-\alpha||_\infty +||\alpha-p^A_i||\right)-\omega_{n,k}(A)\\
    &< & \delta +\omega_{n,k}(A)+\varepsilon -\omega_{n,k}(A)\\
    &= & \delta + \epsilon.
\end{eqnarray*}
Since this is true for all $\varepsilon>0$,
$$ \omega_{n,k}(B)-\omega_{n,k}(A)\leq \delta.$$
Similarly it follows that,
$$ \omega_{n,k}(A)-\omega_{n,k}(B)\leq \delta,$$
and we conclude that,
$$|\omega_{n,k}(A)-\omega_{n,k}(B)|\leq \dH(A,B). $$

\end{proof}

At first sight, this family of valuations is not generated by the max-induced method of $\S 6.1$. We now prove that this valuation cannot be realized as a max-induced valuation.

\begin{thm} Let $k\geq 2$ and $n\in \N$. There is no function $f_k:\R^{n\times n}\rightarrow \R$ such that,
$$\omega_{n,k}(A)= \max_{\alpha\in A} f_k(A), \hspace{1cm} \forall \, A\in \pow(\R^{n\times n}).$$
\end{thm}
\begin{proof}
We will prove this by contradiction. Let $k\geq 2$ and $f_k:\R^{n\times n}\rightarrow \R$ be a function such that $\omega_{n,k}(A)= \max_{\alpha\in A} f_k(A)$. 
	For all $A\subset \R^{n\times n}$ such that $|A|\leq k$, $\omega_{n,k}(A)= 0$. This can be seen from the fact that, if $A=\{\alpha_1,...,\alpha_\ell\}$, with $1 \geq \ell\leq k$, for $p_i=\alpha_i$ for $1\leq i\leq \ell$ and $p_i= \alpha_\ell$ for $\ell< i\leq k$,
    $$ 0\leq \omega_{n,k}(A)\leq \max_{\alpha_j\in A}\min_{1\leq i\leq k}||\alpha_j-p_i||_\infty= 0. $$
    It follows that,
\begin{eqnarray*}
    \omega_{n,k}(A)&=&\max_{\alpha \in A}f_k(\alpha)\\
    &= & \max_{\alpha\in A}\omega_{n,k}(\{\alpha\})=0.
\end{eqnarray*}
Let $A\in \pow(\R^{n\times n})$ such that $|A|>k$. Then, 
$$ 0=\omega_{n,k}(A)= \inf_{\{p_1,...,p_k\}}\max_{\alpha\in A}\min_{1\leq i\leq k}||\alpha-p_i||_\infty.$$
Let $\{p^A_1,...,p^A_k\}$ the minimizing set of $ \omega_{n,k}(A)$. Then, for all $\alpha \in A$ there is $i\in \{1,...,k\}$ such that $||\alpha-p^A_i||_\infty=0$. From this, $A\subset \{p^A_1,...,p^A_k\}$. This yields a contradiction, since $|\{p^A_1,...,p^A_k\}|\leq k$. We conclude that there is no such $f_k$.

\end{proof}

\rem{For all $n\in\N$ and $\ell < k$, $\omega_{n,k}$ and $\omega_{n,\ell}$ are different functions and $\omega_{n,k}\leq \omega_{n,\ell}$. 
\begin{proof}
	First, let us prove that $\omega_{n,k}\leq \omega_{n,\ell}$. For this, let $A\in\pow(\R^{n \times n})$. We observe that, for each set $P_\ell= \{p_1,...,p_{\ell}\}\subset \{p_1,...,p_k\}\R^{n\times n}$,
\begin{eqnarray*}
	\max_{\alpha\in A}\min_{1\leq i\leq k}||\alpha-p_i||_\infty &\leq &\max_{\alpha\in A}\min_{1\leq i\leq \ell}||\alpha-p_i||_\infty\\   
\end{eqnarray*}
It follows that, 
\begin{eqnarray*}
\omega_{n,k}(A) &=& \inf_{\{p_1,...,p_k\}\subset \R^{n\times n}}\max_{\alpha\in A}\min_{1\leq i\leq k}||\alpha-p_i||_\infty \\
 &\leq & \max_{\alpha\in A}\min_{1\leq i\leq \ell}||\alpha-p_i||_\infty\\ 
 &= & \omega_{n,\ell}(A).
\end{eqnarray*}
 Now, let us prove that if $|A|\leq k$, $\omega_{n,k}(A)=0$ and if $|A|>k$, $\omega_{n,k}(A)$: let $A=\{\alpha_1,...,\alpha_j\}\subset \R^{n\times n}$ where all $\alpha_1,...,\alpha_j$ are different and $j\leq k$. Then, by choosing $p_i=\alpha_i$ for $1\leq i \leq j$ and $p_i= \alpha_j$ for $j< i \leq k$, it follows that,
 $$\max_{\alpha\in A} \min_{1\leq i\leq k}||\alpha-p_i||_{\infty}\leq \max_{\alpha\in A }\min_{\alpha'\in A}||\alpha-\alpha'||_\infty= 0. $$
 Since $\omega_{k,n}\geq 0$, it follows that $\omega_{n,k}(A)= 0$. To prove that $\omega_{n,k}(A)>0$ if $|A|>k$ we will proceed by contradiction. Let $A\in \pow(\R^{n\times n})$ such that $|A|>k$ and $\omega_{n,k}(A)= 0$. Let $\varepsilon= \frac{1}{2} \max_{\alpha,\alpha'\in A} ||\alpha -\alpha'||_\infty$. Let $\{p_1,...,p_k\}\subset \R^{n\times n}$ such that,
 $$ \max_{\alpha\in A}\min_{1\leq i\leq k}||\alpha-p_i||_\infty< \varepsilon.$$
 Then, for all $\alpha\in A$ there is an index $1\leq i\leq k$ such that $||\alpha-p_i||_\infty<\epsilon$. Now, notice that, if $||\alpha-p_i||_\infty < \epsilon$, then, for all $\alpha'\in A-\{\alpha\}$,
 $$ ||\alpha'-p_i||_\infty \geq ||\alpha-\alpha'||_\infty-||\alpha-p_i||_\infty > \varepsilon.$$
From this, two different points $\alpha,\alpha'\in A$ cannot share the same point $p_i$ at distance less than $\varepsilon$. Then, for each $\alpha\in A$ we can assign a different index $1\leq i\leq k$. But this gives an injective map from a set $A$ to $\{1,...,k\}$. This holds a contradiction since $|A|>k$.
To conclude, let $A\in \pow(\R^{n\times n})$ be a set such that $|A|=k$ Then,
$$ \omega_{n,k}(A)= 0 < \omega_{n,\ell}(A).$$
Then, $\omega_{n,k}\neq \omega_{n,\ell}.$
 \end{proof}
}

\subsection{Functoriality and characterization results}\label{sub:functcharact}

From Definition \ref{def:filtfunct} we have seen that local filtration functors are not necessarily functorial on the category of 1-Lipschitz maps. This is related to the fact that we are only imposing monotonicity and stability to our valuations. The former assures our constructions are well defined and the latter the stability of the diagrams. But both are not taking into account relation between sets that are not proximity or containment of the sets. We still have to find conditions on valuations that relate them with the functorial category of their induced filtration functor.

For this, we define a partial order on $\R^{n\times n}$ as follows: for all $\alpha,\beta\in \R^{n\times n}$, we say that $\alpha\leq \beta $ if for all $1\leq i,j\leq n$, $\alpha_{ij}\leq \beta_{ij}$. Furthermore, We will say that $\alpha< \beta $ if $\alpha\leq \beta$ and for some $1\leq i,j\leq n$, $\alpha_{ij}<\beta_{ij}$. 

Now we can also define a partial order on $\pow(\R^{n\times n})$. Given $A,B\in\pow(\R^{n\times n})$, we say that $A\leq B$ if for all $\alpha\in A$ there exists $\beta \in B$ such that $\alpha\leq \beta$. We say that $A< B$ if $A\leq B$ and for some $\alpha\in A$, there is $\beta \in B$ such that $\alpha< \beta$.

With this in mind, we can define what it is for a valuation to be increasing.

\begin{defn}
	Let $\nu_n : \pow( \R^{n \times n})\rightarrow \R $ be a valuation. We say that it is \emph{increasing} if for all $A,B\in \pow(\R^{n\times n})$ such that $A\leq B$, $\nu_n(A)\leq \nu_n(B).$
\end{defn}

This condition is independent of the monotonicity condition. It is relating the valuation on sets not through containment but from features of the matrix entrances of each set.

Let $X,Y\in \M$ be spaces such that we can define a 1-Lipschitz, surjective map $f:X\rightarrow Y$. In principle, there is no relation between $\cset_n(X)$ and $\cset_n(Y)$ in terms of containment, since the geometry of $Y$ and $X$ could be quite different. What we can assure is that for all $x_1,...,x_n\in X$,
$$ D^{(n)}_Y(f(x_1),...,f(x_n))\leq D^{(n)}_X(x_1,...,x_n).$$

Since $f$ is surjective, this implies $\cset_n(Y)\leq \cset_n(X)$. From this observation, the following holds:

\begin{prop}
	Let $\nu_n$ be an increasing $n-$valuation. Then, $\Phi^{\nu_n}$ is a functorial filtration over the 1-Lipschitz category. 
\end{prop}

\begin{proof}
	Let $X,Y\in \M$ such that there is a 1-Lipschitz map $f$ from $X$ to $Y$. Let $\sigma \subset X$. This induces a surjective 1-Lipschitz map $f: \sigma \rightarrow f(\sigma)$ considering $\sigma$ and $f(\sigma)$ with their induced subspace metric. Then, $\cset_n(\iota_Y(f(\sigma)))\leq \cset_n(\iota_X(\sigma))$, and,
    $$\Phi^{\nu_n}_X(\sigma)= \nu_n(\cset_n(\iota_X(\sigma)))\leq  \nu_n(\cset_n(\iota_Y(f(\sigma))))= \Phi^{\nu_n}_Y(f(\sigma)).$$
\end{proof}

We already proved that if $f:\R^{n\times n}\rightarrow \R$ is L-Lipschitz, then $\nu^f$ is a L-stable valuation. We can also prove that there is a simple condition that we can impose to f for it to generate an increasing valuation. This condition is for the function to be \emph{increasing}:
$$\forall \alpha,\beta \in \R^{n\times n} \text{ such that } \alpha\leq \beta ,\text{  } f(\alpha)\leq f(\beta).$$

\begin{prop}
	Let $f:\R^{n\times n }\rightarrow \R$ be an increasing function. Then $\nu^f$ is an increasing valuation.
\end{prop}
\begin{proof}
	Let $A,B\in \pow(\R^{n\times n})$ such that $A\leq B$. Then, for all $\alpha\in A$ there exists a $B\beta\in B$ such that $\alpha\leq \beta$. From this 
    $$\nu^f(A)=\max_{\alpha\in A} f(\alpha) \leq \max_{\beta\in B}f(\beta)= \nu^f(B).$$
\end{proof}

 We can prove that the Vietoris-Rips filtration functor is, in some sense, the only increasing 2-local filtration functor. We do not even have to impose stability for this to be true.
\begin{prop}
	Let $\nu_2:\pow(\R^{2\times 2})\rightarrow \R$ be an increasing valuation. Then, there is an increasing function $f:\R\rightarrow \R$ such that $\Phi^{\nu_2}= f\circ R$.
\end{prop}
\begin{proof}
First, we define the matrix $M(r)\in \R^{2\times 2}$ to be,
    $$M(r)= \begin{pmatrix}0 & r \\ r & 0 \end{pmatrix},$$
    for all $r\in \R$.
Let us notice first that given $X\in\M$ and $\sigma\subset X$,
$$ \cset_2(\iota_X(\sigma))=\left\{ M(d_X(x,x')) : x,x' \in \sigma \right\}$$
    Let $r_0= \diam(\iota_X(\sigma))$ and $x,x'\in X$ such that $d_X(x,x')=r_0$. Then, we observe that,
\begin{eqnarray*}
	M(r_0)&\in & \cset_2(\iota_X(\sigma)),\\
    \cset_2(\iota_X(\sigma))&\leq & \left\{M(r) \right\}.
\end{eqnarray*}
Since $\nu_2$ is increasing, $\nu_n(\cset_2(\iota_X(\sigma)))= \nu_n(\{M(r_0)\})$. Then, if we call $f= \nu_n\circ M$, then $\Phi^{\nu_n}_X(\sigma)= f(R_X(\sigma))$.
\end{proof}

We can strengthen our notion of increasing valuation on general $n-$valuations. This strong condition is enough for us to prove several results that connect stability, functoriality and the Vietoris-Rips filtration functor:

\begin{defn}
	We say a $n-$valuation is \emph{strongly increasing} if for all $A,B \in \pow(\R^{n\times n})$, such that for all $\alpha\in A$ and $1\leq i,j\leq n$ there is a matrix $\beta \in B$ such that $\alpha_{ij}< \beta_{ij}$, then $\nu_n(A)\leq \nu_n(B)$.
\end{defn}

Now, given $A\in \pow(\R^{n\times n})$, we define $\max A\in \R^{n\times n}$ to be given by:
$$(\max A)_{ij}= \max_{\alpha\in A}\alpha_{ij}, \hspace{1cm} \forall \, 1\leq i,j\leq n.$$

\begin{thm}
	Let $\nu_n$ be an increasing $n-$valuation. Then the following holds: 
    \begin{enumerate}
    \item For all sets $A\in \pow(\R^{n\times n})$,
    $$ \max_{\alpha\in A} \nu_n(\{\alpha\})\leq \nu_n(A)\leq \nu_n(\{\max A\}).$$ 
    \item If $\nu_n$ is strictly increasing and $1$-stable then,
    $$\max_{\alpha\in A} f_{\nu_n}(\alpha)=\nu_n(\{\max A \}).$$
    \item Furthermore, given $\nu_n$ strictly increasing and $1-stable$, there exists a 1-Lipschitz function $f_{\nu_n}:\R\rightarrow \R$ such that $\Phi^{\nu_n}= f\circ R$.
    \end{enumerate}
\end{thm}
\begin{proof}
(1) Let $A\in \pow(\R^{n\times n})$. Since $\{\alpha\}\subset A\leq \{\max A\}$, it follows that:
$$\max_{\alpha\in A} \nu_n(\{\alpha\})\leq \nu_n(A) \leq \nu_n(\{\max A\}). $$

(2) Let $(\alpha^{(n)})_{n\in \N}$ be a sequence of matrices in $\R^{n\times n}$ such that for all $1\leq i,j\leq n$, $(\alpha^{(n)}_{ij})_n$ is a strictly increasing sequence that converges to $(\max A)_{ij}$. For all $n\in \N$, each entry of the matrix $\alpha^{(n)}$ is dominated by some entrance in $A$ since $\alpha< \max A$. From this, it follows that,
$$\nu_n(A\cup\{\alpha^{(n)}\})\leq \nu_n(A). $$
Now, we know that $\nu_n$ is stable, from which it follows that,
$$ |\nu_n(A)-\nu_n(B)|\leq \dH(A,B).$$
Since $\alpha^{(n)}\rightarrow \max A$ as $n\rightarrow \infty$, 
$$\dH(A\cup\{\alpha^{(n)}\}, A\cup\{\max A \} )\rightarrow 0.$$
Since $A\leq \{\max A\}$, then,
$$\nu_n(\{\max A\}) = \nu_n(A\cup \{\max A\})= \lim_{n\rightarrow \infty} \nu_n(A\cup \{\alpha^{(n)}\})\leq \nu_n(A).$$
Then, $\nu_n(A)= \nu_n(\{\max A\}).$

(3) Let $\nu_n$ be stricly increasing and 1-stable. Let $M(r)= ((1-\delta_{ij})r)_{i,j=1}^n$. Let $f_{\nu_n}:\R\rightarrow \R$ be a function given by:
$$f_{\nu_n}(r)= \nu_{n}(\{M(r)\}), \hspace{1cm} \forall r\in \R.$$
Let $X\in \M$ and $\sigma \subset X$. Let $x,x'\in \sigma$ such that $d_X(x,x')=\diam(\iota_X(\sigma))$. Let $r_0:= \diam(\iota_X(\sigma))$. We observe that, for any $1\leq i,j\leq n$, if $(x^*_1,...,x^*_n) \in X^n$ is such that $x^*_k=x$ if $k\neq j$ and $x^*_j= x'$ then, 
$$D^{(n)}_X(x^*_1,...,x^*_n)_{ij}= d_X(x^*_i,x^*_j)= d_X(x,x')= r_0. $$

Also, we know that for all $\alpha \in \cset_n(\iota_X(\sigma))$ and $1\leq i,j\leq n$, $\alpha_{ij}\leq r_0$ and $\alpha_{ii}=0$. It follows that $\max \cset_n(\iota_X(\sigma))= M(\diam(\iota_X(\sigma))).$ From this, since $\nu_n$ is strictly increasing, 
\begin{eqnarray*}
	\Phi^{\nu_n}_X(\sigma)&=&\nu_n(\cset_n(\iota_X(\sigma)))\\
    &=& \nu_n(M(r_0))\\
    &= &f_{\nu_n}(R_X(\sigma)).
\end{eqnarray*}
Now, we can observe that $r\mapsto M(r)$ is an isometric embedding of $\R$ into $\pow(\R^{n\times n})$. Since $\nu_n$ is 1-stable, which means it is 1-Lipschitz, it follows that $f_{\nu_n}$ is 1-Lipschitz.

$$ \nu_n(\cset_n(\iota_X(\sigma)))= \nu_n(M(\diam(\iota_X(\sigma))))= f_{\nu_n}(R_X(\sigma)).$$

\end{proof}

\section{Computational Examples}\label{sec:exp}

Throughout this section, we write FPS to denote the \emph{farthest point sampling} procedure.\footnote{This sometimes referred to as \emph{sequential max-min sampling} as well.}

\subsection{The $\Phi^{\ult}$ filtration}

First, it is possible for us to point out some remarks:
\begin{itemize}
	\item At time zero, the filtration adds all isosceles triangles with the two longest sides equal. In general, if a subset of a metric space is ultrametric, the whole simplex generated by this subset is added at time zero.
    \item A triangle will be added early in the filtration if it is almost ultrametric (independent of its size) or if it is small enough to have a small ultrametricity.
    \item Isosceles triangles with two shortest sides equal are not ultrametric, so they are added at positive time.
\end{itemize}

To have intuition about how the $\Phi^{\ult}$ filtration functor behaves we programmed the filtration using the javaplex persistent homology package \cite{javaplex}.

To start, we studied the $\Phi^{\ult}$ diagram of $S^1$ with geodesic distance. We computed on an equidistributed sample of 50 points on the unit circle (Figure \ref{fig:ults1}). 

\begin{figure}
	\includegraphics[scale=0.5]{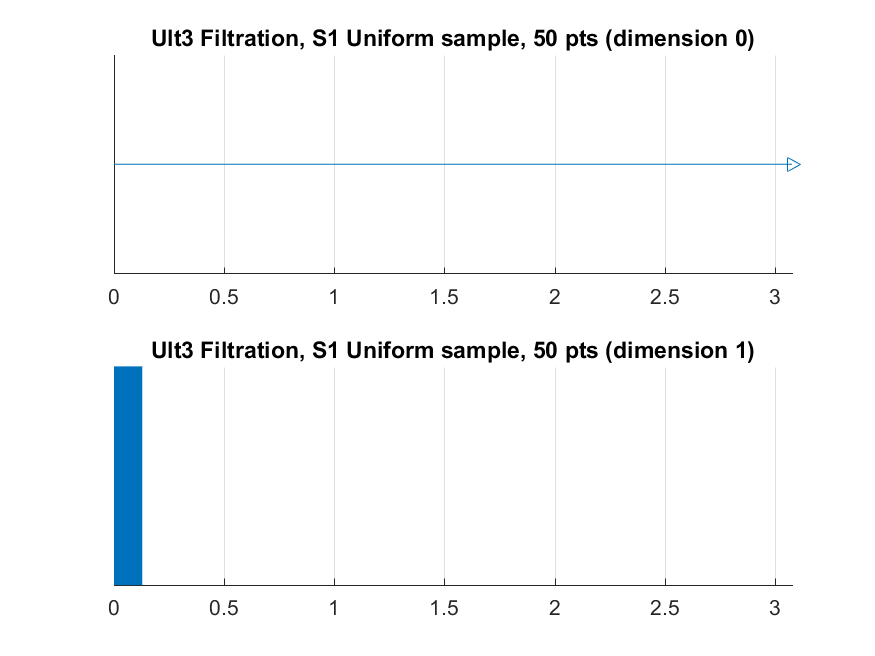}
    \caption{The persistence diagram of a equidistributed 50 point sample of $S^1$. }
    \label{fig:ults1}
    \includegraphics[scale=0.5]{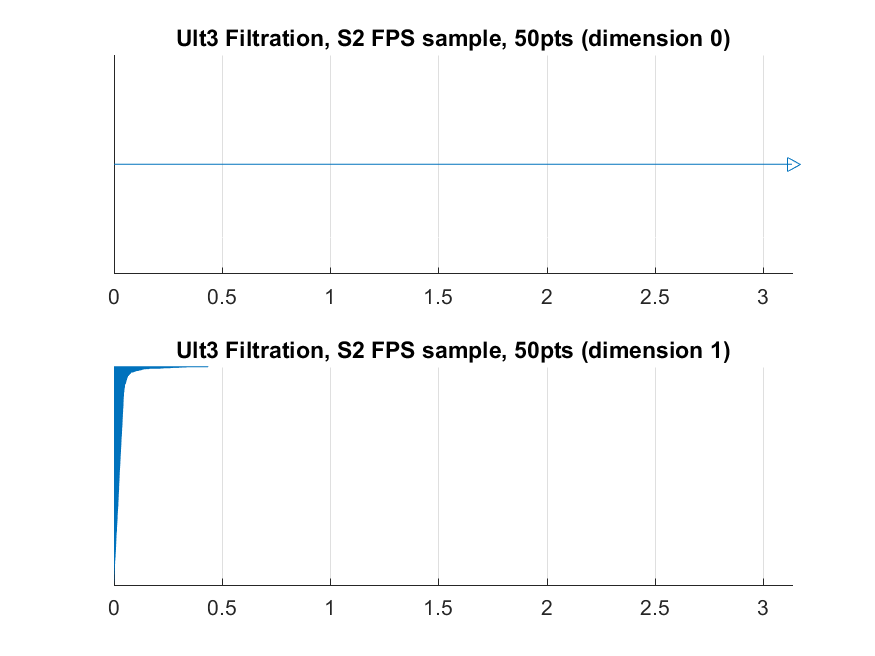}
    \caption{The persistence diagram of a 50 point sample of $S^2$ chosen via farthest point sampling. }
    \label{fig:ults2}
 \end{figure}

Since the 1-skeleton of the space is added at time zero, the 0-diagram consist of just one infinite bar. In the case of dimension 1, all bars die at time $\frac{\pi}{50}.$ To understand this, let us consider the following. After zero, the first radius in which 2-simplexes are added is $\frac{\pi}{50}$. Since we are considering geodesic distance, at this time, all triangles that have an edge of length $\frac{\pi}{50}$ are added. For any cycle, these triangles create a cap that kills the homology generated by it. 

We also computed $\dgm^{\ult}_1(S^2)$ through a FPS 50 point sample of $S^2$ endowed with the geodesic distance (Figure \ref{fig:ults2}). There is a small amount of long bars persisting in this case, that differs from the circle case. To understand this, we decided to compute the persistence diagram of a simpler model of the sphere: a cube (Figure \ref{fig:ultcube}). We considered the set,

$$ C= \left\{ \left((-1)^i\frac{1}{\sqrt{2}},(-1)^j\frac{1}{\sqrt{2}}\right) \bigg| i,j=0,1 \right\}\subset S^2 . $$

We endow $C$ with the induced subspace metric $d_C= d_{\mathbb{S}^2}|_{c\times C}$, where $d_{\mathbb{S}^2}$ is the geodesic distance on $S^2$. 

\begin{figure}
	\includegraphics[scale=0.4]{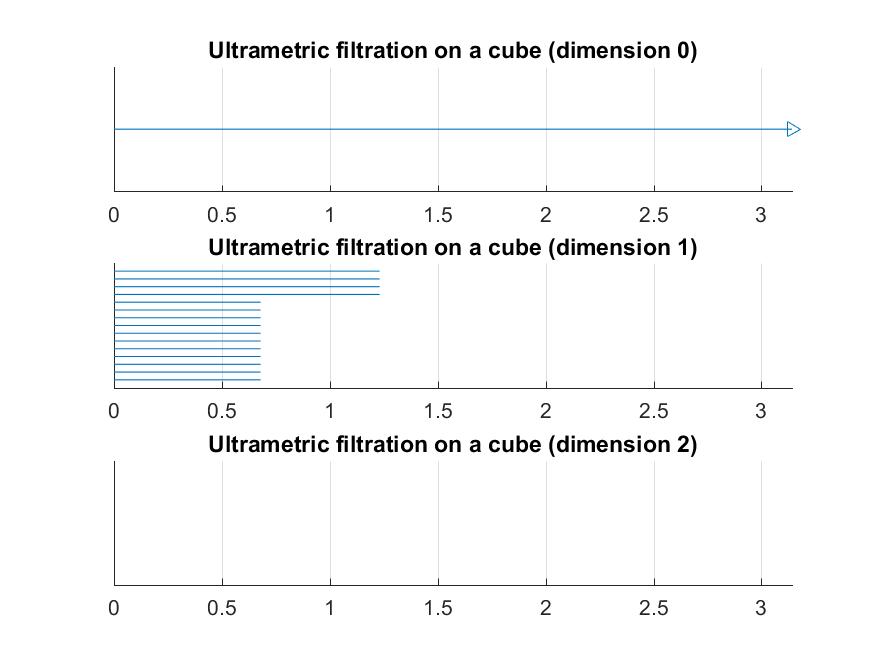}
    \caption{The persistence barcode of a geodesic cube.}
    \label{fig:ultcube}
\end{figure}

\begin{figure}
	\includegraphics[scale=0.25]{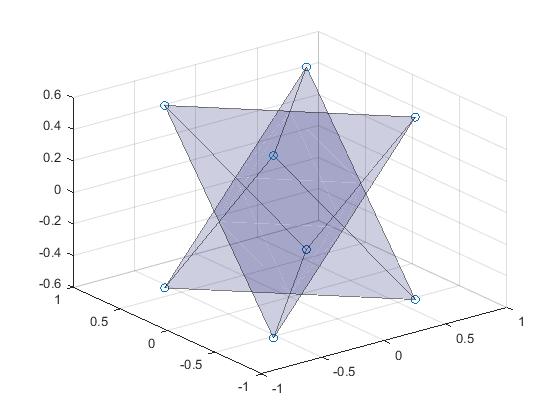}
	\includegraphics[scale=0.25]{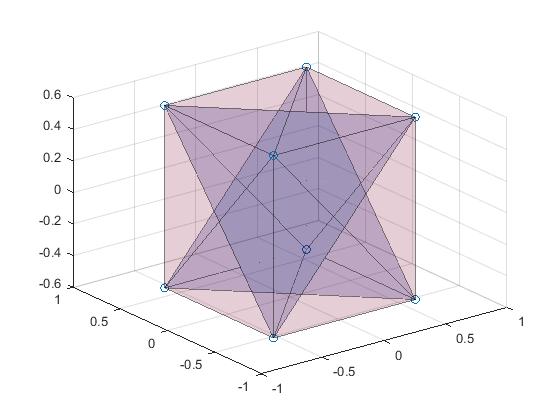}
	\includegraphics[scale=0.25]{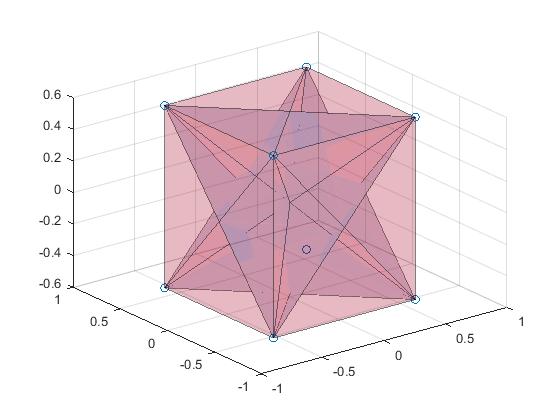}
    \caption{The process of arrival of the 2-dimensional simplexes of the cube on the three relevant times. The 1-dimensional skeleton is not plotted, but it is part of the filtration at all time. } 
    \label{fig:filtcube}
 \end{figure}

 In this discretization we can find a similar pattern, with four bars persisting more than a set of 16 bars in total. We visualize how simplexes are added at each time (Figure \ref{fig:filtcube}). Let us call $\gamma$ the angle that is formed by the center of a regular tetrahedron with respect to two vertices of it (approximately 107 degrees if measured in the 360 scale). We call $\beta$ the angle formed by the center of the cube with respect to two adjacent vertices of it (approximately 70.52 degrees). To be precise, if 
 
 $$O=(0,0,0), \hspace{0.5cm} A_1= \frac{1}{\sqrt(2)}(1,1,1),\hspace{0.5cm}  A_2= \frac{1}{\sqrt(2)}(1,1,-1),\hspace{0.5cm}  A_3= \frac{1}{\sqrt(2)}(1,-1,1),$$ 
 
 then,
 
 $$\gamma= \angle A_2 O A_3, \hspace{1cm} \beta= \angle A_1 O A_2.  $$
 
 The set of all possible distances between points in the cube with the geodesic distance is $\{0,\pi, \gamma, \beta \}$, so ultrametricity of triangles can be written in term of these quantities. It follows the following appearance of 2-simplexes:
 
\begin{itemize}
\item Time 0: two tetrahedron appear, since all faces are equilateral and the ultrametricity of them is zero. 
\item Time $\gamma-\beta$ ($\sim 0.6796$): we get all triangles that form half faces of the cube. 
\item Time $2\pi- \gamma$ ($\sim 1.2301$): we get all missing triangles.
\end{itemize}

By understanding the filtration on a cube, we can explain the 1-dimensional diagram of the sphere in an intuitive way. The 2-simplexes that are added at an early stage are of two types. First, small triangles that are generated by points that are close to each other. Those triplets form the surface of the sphere. The second type are 2-simplexes close to be equilateral. These generate tetrahedrons that do not create 1-dimensional homology. The 2-simplexes that take longer to be added have two diametrically opposite points and one close to be a midpoint between them, since those are the triangles with higher ultrametricity on the sphere. These cycles take longer to be filled.


\subsection{The $\Psi^{\ecc}$ basepoint functor}

If one considers the $\Psi^{\ecc}$ basepoint functor, we see that individual filtrations induced by this functor build up the complex starting at points in the space furthest away from the selected basepoint $x_0$, and working in towards this basepoint. To give more concrete intuition, we developed a program and tested some example finite metric spaces. 

\subsubsection{Implementation details} We used both JavaPlex \cite{javaplex} and Ripser \cite{ripser} for the persistent homology calculations. It is quick to see that all higher dimensional simplices in an eccentricity filtration are determined by the 1-skeleton of that filtration space. Thus, we were able to use Ripser for computing persistent homology in dimensions 1 and higher by putting the distance matrix which Ripser takes in as the matrix of filtration values for all pairs of points in the space. This is useful as Ripser is much faster computationally than JavaPlex, so we can work with larger datasets. Since Ripser adds in all vertices as time 0, and this is not how the eccentricity filtrations work, JavaPlex is used to compute 0-dimensional persistent homology. This is still reasonable as JavaPlex can work quickly for 0-dimensional persistent homology. Also, we make use of the tool in JavaPlex to return the representative of a persistent homology class, which gives useful information, something we will note in the examples later. 
To run the code, there must be an input finite metric space and/or distance matrix. Then, the program will plot the dataset in 3-D. The user must then click on a point within the space. Upon doing so, this point will be selected as the basepoint for the corresponding basepoint eccentricity filtration; the persistent homology of the space using this filtration will be computed, and the persistence barcodes will be plotted to the screen. For visualization help, the plot of the metric space never disappears, and after a basepoint is selected, it is highlighted and all the other points are colored by how close they are to the basepoint; something which helps when considering the eccentricity term in the filtration. After observing the persistence barcodes, the user may then click new points on the original plot to select a new basepoint, and the process will be repeated until the original plot of the metric space is closed. The code is divided up in a way to make it easy to change the filtration while maintaining other functionalities. If one wanted to computationally test a different basepoint filtration functor, to do so they would only have to alter the portion of the code where the eccentricity functor is currently defined.

\subsubsection{A figure 8}
A simple space which gives a good basis of understanding for the eccentricity basepoint filtration is a figure 8. The finite metric space used for this code is a discrete figure 8 with 400 points and using geodesic distance. In the figures below, on the left we see the plot of the metric space, with the red point in the center indicating that this is the selected basepoint. The other points are then colored, with the darker blue points close to the basepoint, and the lighter yellow points being further away. On the right, we see the one-dimensional persistent barcodes of this space. It is not possible to tell, but there are two 1-dimensional persistent intervals; both corresponding to one of the circles.

\begin{figure}[h]
   \includegraphics[width=0.4\textwidth]{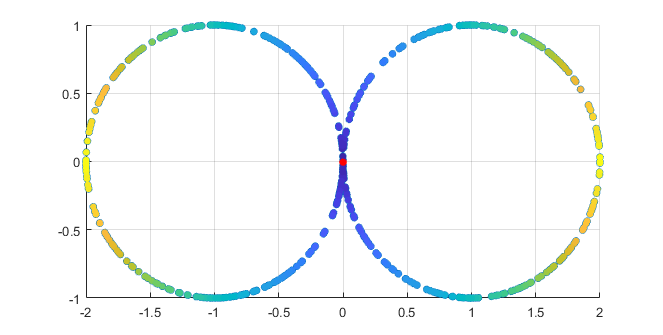}
   \includegraphics[width=0.4\textwidth]{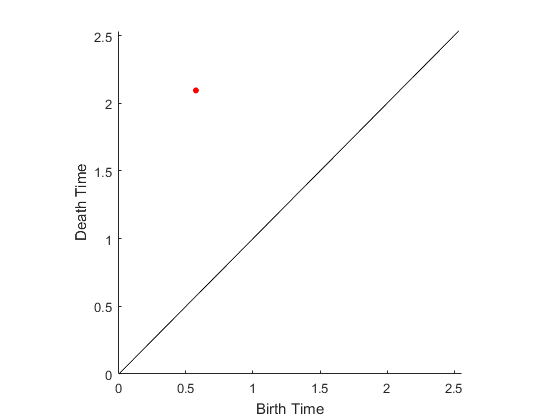}
\end{figure}

This is what we would expect for the 1-dimensional persistent homology of a figure 8 through ``standard'' filtration methods such as Rips, since there are clearly two loops, and of equal size. However, as we move the basepoint around one of the circles, the persistent homology of the new filtration changes. In the next pair of figures, we see what happens when the selected basepoint (highlighted in red) is moved up along one of the circles:

\begin{figure}[h]
   \includegraphics[width=0.4\textwidth]{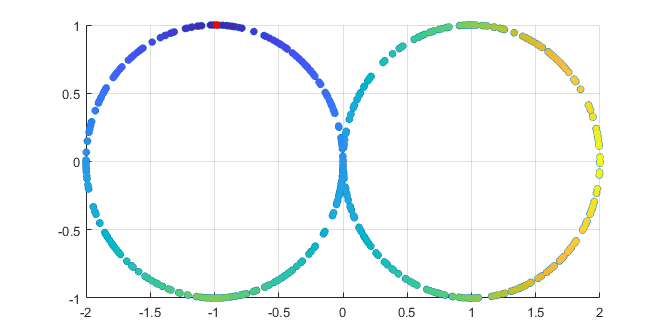}
   \includegraphics[width=0.4\textwidth]{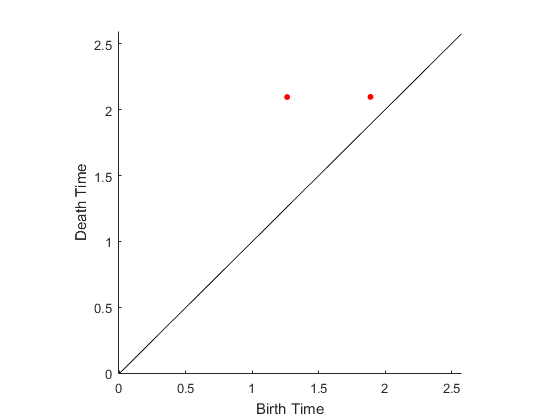}
\end{figure}

We see in this image that one of the 1-dimensional persistence intervals is shrinking, while the other maintains the same length. The eccentricity filtration treats simplices far away from the basepoint like the rips filtration, whereas simplices close to the basepoint are dominated by the eccentricity term. Thus, the shrinking persistence interval corresponds to the loop "closer" to the basepoint; the loop on which the basepoint rests. The unchanging persistence interval then corresponds to the loop separate from the basepoint. As the basepoint is shifted around one of the circles continually further from the center, the 1-dimensional persistence interval corresponding to the loop where the basepoint is completely disappears, leaving only one 1-dimensional persistent interval left. This can be seen in the next pair, where the basepoint is on the outer edge of the figure 8:

\begin{figure}[h]
  \includegraphics[width=0.4\textwidth]{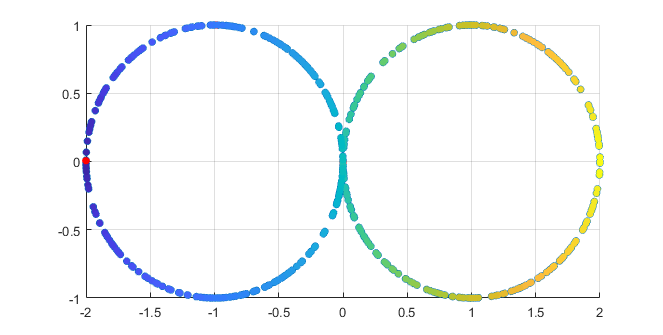}
  \includegraphics[width=0.4\textwidth]{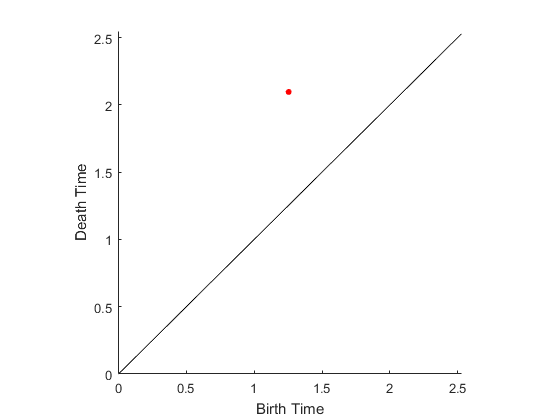}
\end{figure}

\subsubsection{A cat}
Now that a basis of understanding is developed, we proceed with a more complex and real-world example. The space we are dealing with in this example is a high resolution scan of the surface of a cat which is modeled as a finite metric space with geodesic distances. The original scan contained over 27000 points in 3-D, so in order to make it computationally practical we used the built in fps sampling from JavaPlex to select 500 points which filled out the space as well as possible. In the next figure is the original image of the scan, as well as a plot of our 500 selected points.

\begin{figure}[h]
	\includegraphics[width=0.3\textwidth]{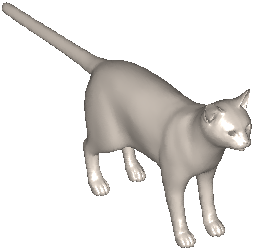}
    \includegraphics[width=0.4\textwidth]{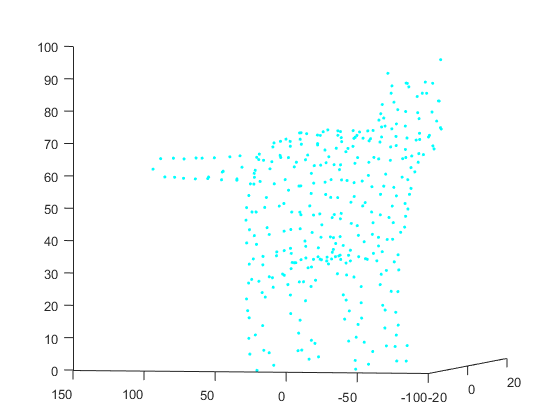}
\end{figure}

The eccentricity filtration can give nice information about the ``protrusions'' of a space, via the 0-dimensional persistence intervals. The figures below demonstrate this.

\begin{figure}[h]
   \includegraphics[width=0.4\textwidth]{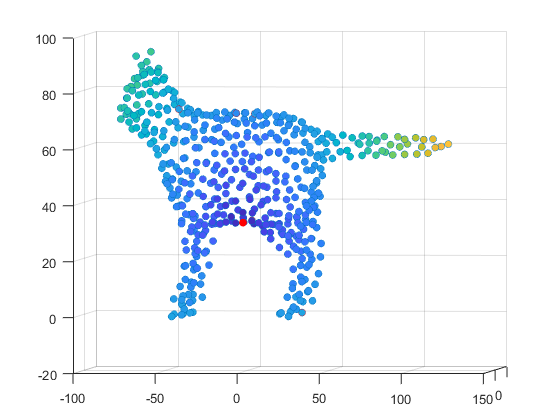}
   \includegraphics[width=0.4\textwidth]{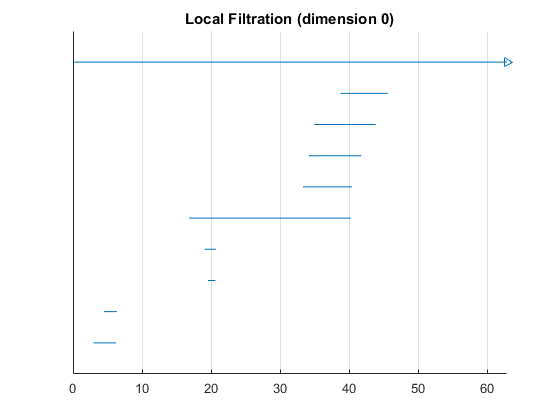}
   \caption{Selected basepoint central}
\end{figure}

A central point on the figure occurs on the belly of the cat. When selecting this (or any other relatively central) point as the basepoint for the eccentricity filtration, one can see all of the "protrusions" of the space as 0-dimensional persistence intervals. The infinitely persisting class always starts at a point which realizes the eccentricity of the basepoint; in this case the tip of the tail. The two short intervals in the bottom left also represents classes from the tail, since at the edge of the tail the diameter term in the filtration is dominant, and thus the filtration behaves similarly to Rips in this region. Then the next two very short intervals correspond to classes starting at the tips of the ears, until they merge with the larger class right above them, which corresponds to a class which originates from a point on the face of the cat. Lastly, the four remaining intervals of similar persistence correspond to the four classes starting at the end of each leg of the cat.

Next, we provide two plots of 1-dimensional persistence barcodes resulting from different selections of basepoint.

\begin{figure}[h]
   \includegraphics[width=0.3\textwidth]{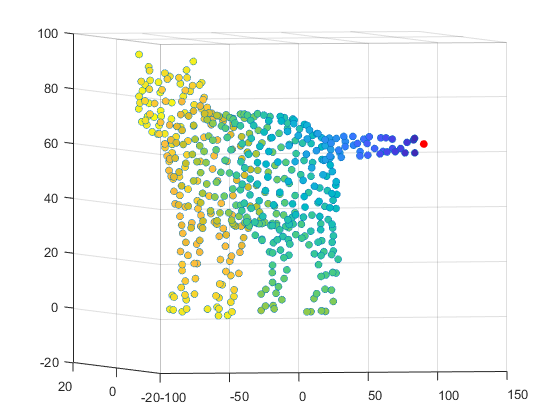}
   \includegraphics[width=0.3\textwidth]{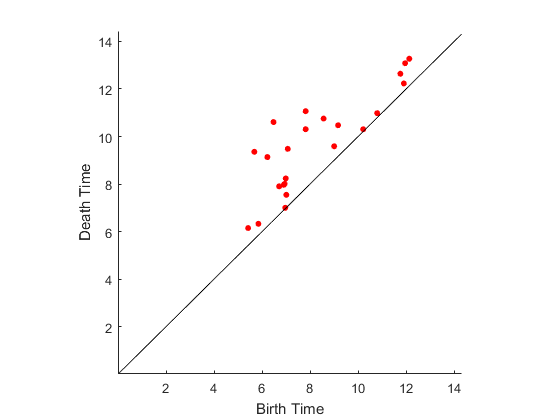}
   \caption{Selected basepoint at tip of tail}
\end{figure}

\begin{figure}[h]
 \includegraphics[width=0.3\textwidth]{bellyclickcat.png}
   \includegraphics[width=0.3\textwidth]{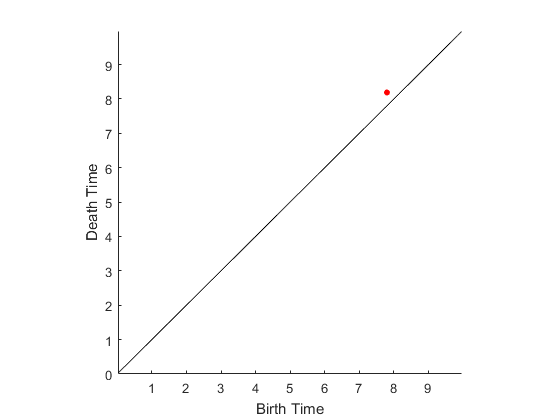}
   \caption{Selected basepoint central}
\end{figure}

In the first, we see the tip of the tail is selected as a basepoint. Since this is at one edge of the space, a large portion of the space opposite this basepoint will behave similarly to rips. The longest persistent interval corresponds to the loop going around the main body of the cat. There are also lots of incredibly short persistent intervals, which mostly correspond to noise as 500 points cannot completely represent a space originating from 27000 points. In the next image, we see the 1-dimensional barcodes when a central point is chosen as the basepoint. As noted, the largest natural loop from this space is the one going around the main body. However, the basepoint lies on this loop, so it won't be realized as a loop in the eccentricity filtration. While choosing a central basepoint is useful for considering 0-dimensional persistence, such a choice of basepoint will more often than not generate significantly fewer 1-dimensional persistent intervals than with a choice of basepoint on the exterior of the space. \\
Next, we want to note that the eccentricity filtration can be adjusted by changing the constant of $\frac{1}{2}$ in front of the eccentricity term. Below, we provide the 1-dimensional barcode with the same central basepoint, but a constant of $\frac{1}{4}$ on the eccentricity term instead. 

\begin{figure}[h]
 \includegraphics[width=0.3\textwidth]{bellyclickcat.png}
   \includegraphics[width=0.3\textwidth]{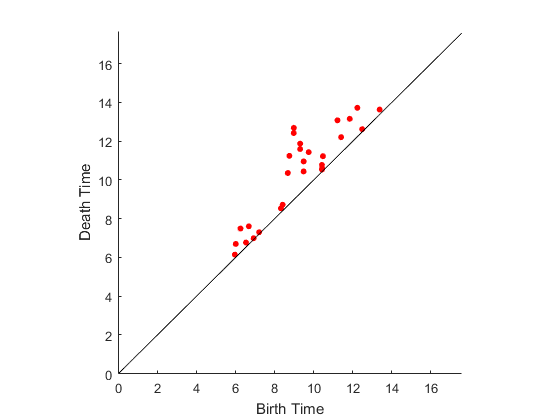}
   \caption{Eccentricity constant 0.25}
\end{figure}

Note that there are many more, and longer, persistence intervals than the same filtration with constant of 0.5 on the eccentricity term. What this constant regulates is effectively how far away from the basepoint to filter the space as if it is the rips filtration. The smaller the constant, the greater the portion on the exterior (relative to the basepoint) of the space is treated similarly to rips. In fact, if this constant is 0 we see that this eccentricity filtration is equivalent to the rips filtration. 


\section{Acknowledgements}
The material in this report emerged from a research project carried out while the authors were in residence at the Institute for Computational and Experimental Research in Mathematics in Providence, RI, during the Summer@ICERM 2017 program \cite{sati17}. The effort was supported by the National Science Foundation under Grants Nos. DMS-1439786, IIS-1422400, and  CCF-1526513. J. \'A. S\'anchez was partially supported by the TDA project conducted by the Center for Research in Mathematics (CIMAT) in Guanajuato, Mexico.

\bibliographystyle{alpha}
\bibliography{biblio}

\begin{thebibliography}{CZCG05}

\bibitem[Bau16]{ripser}
U.~Bauer.
\newblock Ripser, 2016.

\bibitem[BBI01]{burago}
Dmitri Burago, Yuri Burago, and Sergei Ivanov.
\newblock {\em A course in metric geometry}, volume~33.
\newblock American Mathematical Society Providence, RI, 2001.

\bibitem[Car09]{carlsson2009topology}
Gunnar Carlsson.
\newblock Topology and data.
\newblock {\em Bulletin of the American Mathematical Society}, 46(2):255--308,
  2009.

\bibitem[CM10]{clust-um}
Gunnar Carlsson and Facundo M{\'e}moli.
\newblock Characterization, stability and convergence of hierarchical
  clustering methods.
\newblock {\em J. Mach. Learn. Res.}, 11:1425--1470, August 2010.

\bibitem[CMS16]{tree-nips}
Samir Chowdhury, Facundo M{\'e}moli, and Zane~T Smith.
\newblock Improved error bounds for tree representations of metric spaces.
\newblock In {\em Advances in Neural Information Processing Systems}, pages
  2838--2846, 2016.

\bibitem[CZCG05]{carlsson2005persistence}
Gunnar Carlsson, Afra Zomorodian, Anne Collins, and Leonidas~J Guibas.
\newblock Persistence barcodes for shapes.
\newblock {\em International Journal of Shape Modeling}, 11(02):149--187, 2005.

\bibitem[EH10]{edels-book}
Herbert Edelsbrunner and John Harer.
\newblock {\em Computational topology: an introduction}.
\newblock American Mathematical Soc., 2010.

\bibitem[EM14]{edelsbrunner2014persistent}
Herbert Edelsbrunner and Dmitriy Morozov.
\newblock Persistent homology: theory and practice.
\newblock 2014.

\bibitem[FJ16]{Frosini2016}
Patrizio Frosini and Grzegorz Jab{\l}o{\'{n}}ski.
\newblock Combining persistent homology and invariance groups for shape
  comparison.
\newblock {\em Discrete {\&} Computational Geometry}, 55(2):373--409, Mar 2016.

\bibitem[Fro92]{frosini1992measuring}
Patrizio Frosini.
\newblock Measuring shapes by size functions.
\newblock In {\em Intelligent Robots and Computer Vision X: Algorithms and
  Techniques}, volume 1607, pages 122--134. International Society for Optics
  and Photonics, 1992.

\bibitem[Gro07]{gromov}
Mikhail Gromov.
\newblock {\em Metric structures for Riemannian and non-Riemannian spaces}.
\newblock Springer Science \& Business Media, 2007.

\bibitem[M\'17]{tripods}
Facundo M\'{e}moli.
\newblock A distance between filtered spaces via tripods.
\newblock {\em arXiv preprint arXiv:1704.03965}, 2017.

\bibitem[M{\'e}m12]{dgh-props}
Facundo M{\'e}moli.
\newblock Some properties of {G}romov--{H}ausdorff distances.
\newblock {\em Discrete \& Computational Geometry}, 48(2):416--440, 2012.

\bibitem[Rob99]{robins1999towards}
Vanessa Robins.
\newblock Towards computing homology from finite approximations.
\newblock In {\em Topology proceedings}, volume~24, pages 503--532, 1999.

\bibitem[SS03]{phylo}
Charles Semple and Mike~A Steel.
\newblock {\em Phylogenetics}, volume~24.
\newblock Oxford University Press on Demand, 2003.

\bibitem[Sum17]{sati17}
Summer@ICERM.
\newblock Topological data analysis, 2017.

\bibitem[TVJA11]{javaplex}
Andrew Tausz, Mikael Vejdemo-Johansson, and Henry Adams.
\newblock Javaplex: A research software package for persistent (co) homology.
\newblock {\em Software available at http://code. google. com/javaplex}, 2,
  2011.

\end{thebibliography}

\newpage

\end{document}